\documentclass[a4paper,12pt]{article}

\usepackage{amsmath,amssymb,amsthm}
\usepackage{amsfonts}
\usepackage{multicol}
\usepackage{color}

\definecolor{mygreen}{RGB}{0,160,50}

\newtheorem{theorem}{Theorem}[section]
\newtheorem{lemma}[theorem]{Lemma}

\newtheorem{corollary}[theorem]{Corollary}
\newtheorem{e-definition}[theorem]{Definition\rm}
\newtheorem{remark}[theorem]{Remark}

\def\ignore#1{}
\def\eq{\begin{equation}}
\def\en{\end{equation}}
\def\eqa{\begin{eqnarray}}
\def\ena{\end{eqnarray}}
\def\eqs{\begin{eqnarray*}}
\def\ens{\end{eqnarray*}}

\def\re{{\mathbb R}}

\def\non{\nonumber}

\def\s{\sigma}
\def\l{\lambda}
\def\d{\delta}
\def\Ref#1{(\ref{#1})}

\def\Eq{\ =\ }
\def\Le{\ \le\ }
\def\Z{{\mathbb{Z}}}
\def\h{\eta}

\def\n{\nu}
\def\e{\varepsilon}
\def\f{\varphi}

\def\NN{{\mathcal N}}
\def\nin{\noindent}

\def\law{{\mathcal L}}
\def\JJ{{\mathcal J}}

\def\AA{{\mathcal A}}
\def\half{{\tfrac12}}
\def\quarter{{\tfrac14}}
\def\eighth{{\tfrac18}}
\def\third{{\tfrac13}}
\def\pr{{\mathbb P}}
\def\ex{{\mathbb E}}
\def\Def{{\ \,:=\ \,}}

\def\D{\Delta}
\def\uii{^{(i)}}

\def\Blb{\left\{}
\def\Brb{\right\}}
\def\giv{\,|\,}
\def\Giv{\,\Big|\,}

\def\uj{^{(j)}}

\def\dtv{d_{{\rm TV}}}

\def\m{\mu}

\def\msk{\medskip}
\def\L{\Lambda}
\def\S{\Sigma}
\def\Blm{\left|}
\def\Brm{\right|}

\def\ej{e^{(j)}}
\def\eii{e^{(i)}}
\def\sjd{\sum_{j=1}^d}
\def\sid{\sum_{i=1}^d}
\def\Blb{\left\{}
\def\Brb{\right\}}
\def\Bl{\left(}
\def\Br{\right)}
\def\tr{{\rm Tr}\,}
\def\th{\theta}

\def\nud{_n^\d}
\def\nuh{_n^\h}

\def\DN{\mathcal {DN}}

\def\leqn{\lefteqn}

\def\g{\gamma}

\def\Lbar{\overline{\L}}
\def\lmax{\l_{{\rm max}}}
\def\lmin{\l_{{\rm min}}}

\def\bt{t}

\def\ut{^{(2)}}
\def\ui{^{(1)}}
\def\KKS{{\cal K}_\S}
\def\ts{{\tilde s}}

\def\G{\Gamma}

\def\gbars{{\bar\g}(\s^2)}

\ignore{
\addtolength{\textwidth}{1.5cm}
\addtolength{\hoffset}{-1cm}
\addtolength{\textheight}{1.5cm}
\addtolength{\voffset}{-1cm}
}
\def\nti{n\to\infty}
\def\siim{\sum_{i=1}^m}
\def\uh{^{(3)}}

\def\lla{\alpha}

\def\hS{{\widehat\S}}
\def\hA{{\widehat A}}
\def\vva{{u}}
\def\vvb{{\tilde u}}
\def\z{\zeta}
\def\te{{\tilde\e}}
\def\Rh{\rho}
\def\tR{{\widetilde R}}
\def\FF{{\cal F}}
\def\ccc{\g}
\def\nS#1{|#1|_{\S}}
\def\nl#1{|#1|_1}
\def\sld{\sum_{l=1}^d}
\def\ccc{\xi}
\def\td{{\tilde \d}}
\def\VVV{S}
\def\vvs{s}
\def\eth{\th}
\def\Sp{{\rm Sp}}
\def\lbar{{\bar\l}}
\def\tY{{\widetilde Y}}
\def\sjn{\sum_{j=1}^n}
\def\var{{\rm Var\,}}
\def\cov{{\rm Cov\,}}
\def\skn{\sum_{k=1}^n}
\def\smh{\psi}
\def\tn{{\tilde n}}
\def\wA{{\widetilde A}}
\def\wS{{\widetilde \S}}
\def\moda{\|A\|}
\def\vnew{\nu}
\def\ABA{{\widetilde\AA}}

\def\adb#1{#1}
\def\adbg#1{#1}
\def\adbn#1{#1}

\def\adbr{}
\def\adbb{}

\begin{document}

\title{\hbox{Multivariate approximation in total variation, II:} discrete normal approximation}

\author{
\renewcommand{\thefootnote}{\arabic{footnote}}
A. D. Barbour\footnotemark[1],
\ M. J. Luczak\footnotemark[2]
\ 
 \& A. Xia\footnotemark[3]
\\
Universit\"at Z\"urich, Queen Mary University of London \\ \& University of Melbourne
}

\footnotetext[1]{Institut f\"ur Mathematik, Universit\"at Z\"urich, Winterthurertrasse 190, CH-8057 Z\"urich;
e-mail {\tt a.d.barbour@math.uzh.ch}.  Work begun while ADB was Saw Swee Hock 
Professor of Statistics at the National
University of Singapore, carried out in part at the University of Melbourne and at Monash University, and
supported in part by Australian Research Council Grants Nos DP120102728, DP120102398, DP150101459 and DP150103588}
\footnotetext[2]{School of Mathematical Sciences, Queen Mary University of London, Mile End Road, London E1 4NS, UK; 
e-mail {\tt m.luczak@qmul.ac.uk}.
Work carried out in part at the University of Melbourne, and supported by an EPSRC Leadership Fellowship, 
grant reference EP/J004022/2, 
and in part by Australian Research Council Grants Nos DP120102398 and DP150101459.}
\footnotetext[3]{School of Mathematics and Statistics, University of Melbourne, Parkville, VIC 3010, Australia;
e-mail {\tt a.xia@ms.unimelb.edu.au}.
Work supported in part by Australian Research Council Grants Nos DP120102398 and DP150101459.}
\maketitle
%} \footnotetext[4]
%  \\ ADB was supported in part by Australian Research Council Grants Nos DP120102728 and DP120102398}

\begin{abstract}
The paper applies the theory developed in Part~I to the discrete normal approximation in 
total variation of random vectors in~$\Z^d$.
We illustrate the use of the method for  sums of independent integer valued
random vectors,  and for random
vectors exhibiting an exchangeable pair.  We conclude with an application to random
colourings of regular graphs.
\end{abstract}
 
{\it Running head:}  Multivariate approximation

\section{Introduction}\label{Introduction}
\setcounter{equation}{0}
In Theorem~4.8 of Barbour, Luczak \& Xia (2017) (Part~I), we establish bounds on the total
variation distance between the distribution of a random element $W \in \Z^d$ and the
equilibrium distribution of a suitably chosen Markov population process~$X_n$.
In this paper, we show that the bounds are of order $O(n^{-1/2}\log n)$ as $\nti$ if
$W \sim \DN_d(nc,n\S)$, for any $c\in\re^d$ and positive definite symmetric $d\times d$ matrix~$\S$,
where the discrete normal distribution $\DN_{d}(nc,n\S)$ is obtained from $\NN_d(nc,n\S)$ 
by assigning the probability of the $d$-box 
\[
    [i_1 - 1/2,i_1 + 1/2) \times\cdots \times [i_d-1/2,i_d+1/2)
\]
to the integer vector $(i_1,\ldots,i_d)$, for each $(i_1,\ldots,i_d) \in \Z^d$.  
From this, we deduce bounds for the discrete normal approximation of any random $d$-vector~$W$.

To state Theorem~4.8 of Part~I in the form that we shall need, we let $c\in\re^d$ be arbitrary,
and $A,\s^2 \in \re^{d\times d}$ be such \adbb{that~$A$ is a Hurwitz matrix (all its eigenvalues 
have negative real parts),} and that~$\s^2$ is positive definite and symmetric.  We let~$\S$ denote the positive
definite solution of the continuous Lyapunov equation
\eq\label{ADB-Sigma-eqn}
    A\S + \S A^T + \s^2 \Eq 0;
\en
for example, if $A = -I$, then $\S = \half\s^2$.  We define an associated norm $\nS{x} := \sqrt{x^T\S^{-1}x}$.
We then define an operator~$\ABA$ acting on
functions $h\colon \Z^d \to \re$ by
\eq\label{ADB-approx-gen}
  \ABA_n h(w) \Def  \frac n2 \tr(\s^2\D^2h(w)) + \D h^T(w)A(w-nc), \quad w\in\Z^d,
\en
where
\[
   \D_j h(w) \Def h(w+\ej) - h(w);\quad \D^2_{jk} h(w) \Def \D_j(\D_k h)(w),\quad 1\le j,k \le d.
\]
For $f\colon \Z^d\to\re$, we also define
\eq\label{ADB-norm-notation}
   \|f\|^\S_{n\h,\infty} \Def \max_{\nS{X - nc} \le n\h} |f(X)|,
\en
with~$nc$ implicit.  We then write $\|\D h\|^\S_{n\h,\infty}$ and
$\|\D^2 h\|^\S_{n\h,\infty}$ for $\|f\|^\S_{n\h,\infty}$, when $f(X) = \max_{1\le j\le d}|\D_j h(X)|$
and $f(X) = \max_{1\le j,k\le d}|\D^2_{jk} h(X)|$, respectively.
For a matrix~$M$, we let~$\|M\|$ denote its spectral norm; if it is positive definite and symmetric, 
we let $\lmax(M)$ and~$\lmin(M)$ denote its largest and smallest eigenvalues, $\Rh(M)$ their ratio
\adbb{and $\Sp'(M) := \{\lmin(M),\lmax(M),d^{-1}\tr(M)\}$.}

\begin{theorem}\label{ADB-first-approx-thm}
\adbr{
Given any $c,A$ and~$\s^2$ as above,
% eigenvalues all with negative real parts, and~$\s^2$ being positive definite, 
there exists an associated sequence of Markov population processes $(X_n,\,n\ge1)$, 
whose restriction~$X\nud$ to the $n\d$-ball in~$\nS{\cdot}$ \adbb{with centre~$nc$}, 
for~$\d \le \lmin(\s^2)/(8\|A\|)$,
has equilibrium distribution~$\Pi\nud$ concentrated \adbb{near~$nc$, which is almost the same
for all~$\d$.}
% with the following properties.  Let $\gbar$ and~$\S$ be as in 
% \Ref{ADB-Lambda-def} and~\Ref{ADB-Sigma-eqn}.} 
The closeness of~$\law(W)$ in total variation to~$\Pi\nud$, for any random vector~$W$ in $\Z^d$,
can be checked as follows. For $\td_0 = \min\{3,\lmin(\s^2)/(8\|A\|\sqrt{\lmax(\S)})$,
and for any~$v > 0$ and $0 < \d' < \half\td_0$, there exist constants  $C_{\ref{ADB-first-approx-thm}}(v,\d')$
and $n_{\ref{ADB-first-approx-thm}}(v,\d')$, which are \adbb{continuous
functions of~$v$, $\d'$, $\|A\|/\Lbar$, $\Sp'(\s^2/\Lbar)$ and $\Sp'(\S)$}, where $\Lbar := d^{-1}\tr(\s^2)$, 
but not of~$n$, with the following property:
if, for some $v>0$, $0 < \d' < \half\td_0$, $n \ge n_{\ref{ADB-first-approx-thm}}(v,\d')$
and  $\e_1,\e_{20},\e_{21},\e_{22} > 0$,
$$
\begin{array}{rl}
 {\rm (i)}& \ex\nS{W -nc}^2 \Le dvn;\\[1ex]
 {\rm (ii)}& \dtv(\law(W),\law(W+\ej)) \Le \e_1, \mbox{ for each } 1\le j\le d;\\[1ex]
 {\rm (iii)}&  |\ex\{\ABA_n h(W)I[\nS{W-nc} \le n\d'/3]\}| \\
          &\qquad\Le \Lbar\,(\e_{20}\|h\|^\S_{n\td_0/4,\infty} + \e_{21}n^{1/2}\|\D h\|^\S_{n\td_0/4,\infty}
                + \e_{22}n\|\D^2 h\|^\S_{n\td_0/4,\infty}),
\end{array}
$$
for all $h\colon \Z^d\to\re$,  where $\ABA_n$ is as defined in~\Ref{ADB-approx-gen},
then, for any~$\d$ such that $2\d' \le \d \le \td_0$, 
\[
    \dtv(\law(W),\Pi\nud) \Le C_{\ref{ADB-first-approx-thm}}(v,\d')
           (d^3n^{-1/2} + d^4\e_1 +\e_{20} + d^{1/4}\e_{21} +d^{1/2}\e_{22})\log n .
\]
}
\end{theorem}

The accuracy of the approximation, for fixed $c$, $A$ and~$\s^2$, is thus of order
$O(\log n\{n^{-1/2} + \e_1 + \e_{20} + \e_{21} + \e_{22}\})$, and is determined by
how small the $\e$-quantities are.  In Section~\ref{examples}, we give examples to show that they
can all be of order~$O(n^{-1/2})$, giving an overall bound of order $O(n^{-1/2}\log n)$.
The constant $C_{\ref{ADB-first-approx-thm}}$ and the quantities $1/\td_0$ and $d^{-1}\tr(\s^2)$ depend
on $A$ and~$\s^2$ in such a way that they do not grow with increasing dimension~$d$, 
provided that the spectral norm of~$A$ and the eigenvalues of~$\s^2$ and~$\S$
remain bounded away from zero and infinity;  more detail is given in Part~I.
\adbb{Note, however, that~$n$ appears in the definition of~$\ABA_n$ only as a product with~$\s^2$,
and so can be chosen to prevent $\tr(\s^2)$ and~$\tr(\S)$ becoming large.}
Note also that the equilibrium distribution~$\Pi\nud$ remains the same if both $A$ and~$\s^2$
are multiplied by a common factor~$a > 0$ --- this merely reflects a new choice of time scale ---
but the operator~$\ABA_n$ is multiplied by~$a$.  The factor~$d^{-1}\tr(\s^2))$ on the
right hand side of the inequality in Condition~(iii) of Theorem~\ref{ADB-first-approx-thm}
ensures that the \adbb{constant~$C_{\ref{ADB-first-approx-thm}}(v,\d')$} 
is the same for all choices of~$a$.
  
The remainder of this paper completes two tasks.  The first is to show that,
if $W \sim \DN_d(nc,n\S)$, then Conditions (i)--(iii) of Theorem~\ref{ADB-first-approx-thm}
are satisfied with all the $\e$-quantities of order $O(n^{-1/2})$.  
As a result, $\Pi\nud$ in the above theorem can be replaced by $\DN_d(nc,n\S)$,
giving the desired method of proving discrete normal approximation.
The second is to show that the theorem can be applied in reasonable generality,
yielding good rates of approximation. Note that there are
many pairs $(A,\s^2)$ that correspond to the same~$\S$, and the flexibility of having many
pairs $(A,\s^2)$ to use when approximating a single discrete normal distribution
$\DN_d(nc,n\S)$ represents a real advantage.

\ignore{
Note also that the expectation $\ex\{\AA_n h_B(W)I\nud(W)\}$ in Condition~(iii)
can be replaced by its analogue $\ex\{\AA_n h_B(W)I[|W-nc| \le n\d'/3]\}$, with truncation outside a 
Euclidean ball, without much extra error, provided that~$\d'$ is such that 
$B_{\d'/3}(c) := \{w\colon |w-c| \le \d'/3\} \subset B_{\d/3,\S}(c)$.  This follows from
the assumption that $\ex\nS{W-nc}^2 \le vn$, and from the bounds obtained on the first and second differences
of the functions~$h_B$; see Section~\ref{5.5-1.1}. 
}

\adbr{
The structure of the paper is as follows.  
A brief taste of the results to be obtained is given in Section~\ref{prelim}. 
In Section~\ref{ADB-discrete-normal}, the main
discrete normal approximation, Theorem~\ref{ADB-DN-approx-thm}, is established,
giving \adbb{two} conditions to be checked in order to conclude discrete normal approximation
in total variation.  If a `linear regression pair' can be found, these conditions
can be substantially simplified; we give a corresponding result in Theorem~\ref{AX-exch-thm} 
of Section~\ref{regression-property}.
This theorem is applied, in Section~\ref{examples}, to sums of independent random vectors, 
and then in the more
general context of exchangeable pairs, as developed in Stein~(1986).
We conclude with an application to the joint distribution of the numbers
of monochrome edges in a graph colouring problem.  
A number of proofs that involve lengthy calculations are deferred to Section~\ref{appendix}.}
\adbb{The form of Theorem~\ref{ADB-DN-approx-thm} also lends itself to use under assumptions of local
 dependence.}

\subsection{Illustration}\label{prelim}
\adbr{
Theorem~\ref{ADB-DN-approx-thm} is somewhat forbidding.  Before going into detail, 
we give a simple corollary of the theorem in the context of exchangeable pairs
having the approximate linear regression property, and sketch
an example. 
}

Suppose that~$(W,W')$ is a pair of random integer valued $d$-vectors, defined
on the same probability space, such that the pairs $(W,W')$ and $(W',W)$ have the same distribution.  
Assume that $\ex\{|W|^3\} < \infty$, and write $\m := \ex W$.
Let~$\xi$ denote the difference $W'-W$, so that $\ex\xi = 0$, and set $\s^2 := \ex\{\xi\xi^T\}$,
assumed positive definite, and $\chi := \ex\{|\xi|^3\}$.  Assume that,
for some $n>0$ and for some Hurwitz matrix~$A \in \re^{d\times d}$ with spectral norm~$\|A\|$, 
% all of whose eigenvalues have negative real parts, 
we have
\eq\label{prelim-Exch-props}
   \adbr{\begin{array}{rl}
   \ex\{\xi \giv W\} &\Eq n^{-1}A(W-\m) + \{\|A\|/n\}^{1/2}R_1(W); \\[1ex]
   \s^2(W) &\Def \ex\{\xi\xi^T \giv W\}.   % \Eq \s^2 + R_2(W).
   \end{array}}
\en 
Clearly, $\ex\{R_1(W)\} = 0$. % \adbn{and $\ex\{R_2(W)\} = 0$}.
% Let~$\S$ be the positive definite solution of $A\S + \S A^T + \s^2 = 0$, and write 
% $\s^2_\S := \S^{-1/2}\s^2\S^{-1/2}$. Define $\lla_1 := \half\lmin(\s^2_\S)$ and 
Write $L := (\|A\|/n)^{1/2}\chi\{\tr(\s^2)\}^{-3/2}$, \adbr{let~$\S$ be the solution 
to~\Ref{ADB-Sigma-eqn},} and assume that 
\[
    \{\ex|\S^{-1/2}R_1(W)|^3\}^{1/3} \Le %\eighth\adbn{(\lla_1\tr(\s^2\S^{-1})/\moda)}^{1/2},
            \frac{\lmin(\s^2)}{8\lmax(\S)} \sqrt{\frac d{2\moda}}\,.
\]
% where~$\S$ is the positive definite solution of $A\S + \S A^T + \s^2 = 0$.
Let~$\JJ$ be the set of $d$-vectors such that $q^J := \pr[\xi=J] > 0$.
Suppose that $\JJ$ is finite, and that each of the coordinate vectors~$\ej$, $1\le j\le d$,
can be obtained as a (finite) sum of elements of~$\JJ$.
For $Q^J(W) := \pr[\xi = J \giv W]$, we set
\[
  \vva^J \Def (q^J)^{-1} \ex|Q^J(W) - q^J|,
\] 
and $\vva^* := \max_{J\in\JJ}\vva^J$.

\begin{theorem}\label{Dec-exch-thm}
Under the above circumstances, there exist constants~$n_0$ and~$C$, depending on $d$, $\s^2$, $\JJ$ and~$A$,
such that, if $n \ge n_0$, we have
\eqs
   d_{TV}\bigl(\law(W),\DN_{d}(\mu,n\S)\bigr) 
  &\le&  C \log n\bigl\{L(1 + n^{1/2}\vva^*)  + \ex|R_1(W)| \bigr\}.
\ens
\end{theorem}

The key elements in the bound are~$L$, which is the analogue of the Lyapunov ratio
appearing in the Berry--Esseen error bound,
$\vva^*$, which can often be shown to be small by a variance calculation, and the
inaccuracy of the linear regression~\Ref{prelim-Exch-props}, expressed by $\ex|R_1(W)|$.
In examples such as the one that follows, the resulting bound is of order $O(n^{-1/2}\log n)$.
The theorem can be deduced from Theorem~\ref{AX-exch-thm}, Lemma~\ref{Exch-dtv-lemma} and
Corollary~\ref{Dec-useful-bounds}.

As an example, 
suppose that~$G_n$ is an $r$-regular graph on~$n$ vertices.
Let the vertices be coloured independently, each with one of~$m$ colours, the
probability of choosing colour~$i$ being~$p_i > 0$, $1\le i\le m$.  Let~$N_i$ denote the
number of vertices having colour~$i$, and let~$M_i$ denote the number of edges joining
pairs of vertices that both have colour~$i$.  We are interested in approximating the
joint distribution of
\[
    W \Def (M_1,\ldots,M_m,N_1,\ldots,N_{m-1})  \ =:\ (W_1,\ldots,W_m,W_{m+1},\ldots,W_{2m-1}),
\]
when~$n$ becomes large, while $r$, $m$ and $p_1,\ldots,p_m$ remain fixed; the detailed
structure of~$G_n$ does not appear in the approximation.  Multivariate normal approximation
in a smooth metric was proved by Rinott \& Rotar~(1996), and in the convex
sets metric by Chen, Goldstein \& Shao~(2011, pp. 333--334), both with error of order
$O(n^{-1/2}\log n)$.  Theorem~\ref{Dec-exch-thm} shows that the same order of error actually holds in
total variation, provided that $m \ge 3$; the details are given in Section~\ref{Dec-monochrome}.
For $m=2$, the distribution of~$W$ is concentrated on a sub-lattice of \adbr{$\Z^{3}$,} so that
discrete normal approximation is not good (but it can be deduced for the pair $(M_1,N_1)$).  
The exchangeable pair is constructed by realizing~$W$ from a random colouring of the vertices,
and then randomly re-colouring one of the vertices to give~$W'$.  
The resulting regression is exact, implying that $R_1(w)= 0$ for all~$w$. The set~$\JJ$ is fixed and finite,
so that $L = O(n^{-1/2})$, and, for each~$J$, $\ex(Q^J(W) - q^J)^2$ can simply be shown
to be of order~$O(n^{-1})$ --- the calculation is as for the variance of a sum of~$n$ very weakly
dependent indicators.  If $m\ge 3$, each coordinate vector~$\ej$, $1\le j\le 2m-1$, can be
obtained as a sum of elements of~$\JJ$, but this cannot be done if $m=2$.  The analogous problem,
in which the proportions of vertices of each colour are held (almost) fixed, but randomly
assigned to the vertices, can be treated in much the same way.  The exchangeable pair is
obtained by swapping the colours of two vertices, and the treatment of $\ex(Q^J(W) - q^J)^2$
becomes a little messier.

\section{Discrete normal approximation}\label{ADB-discrete-normal}
\setcounter{equation}{0}
In this section, we show that Theorem~\ref{ADB-first-approx-thm} can be used to 
establish approximation by distributions from the discrete normal family. To do so,
we need first to establish properties of distributions in the family that are related
to the conditions of Theorem~\ref{ADB-first-approx-thm}. 
We always assume that $n \ge d^4$.

We first note the following simple lemma, proved in Section~\ref{L5.1}, in which 
moments of the discrete normal random variable $W \sim \DN_d(nc,n\S)$ are bounded by
expressions similar to those of $\NN_d(nc,n\S)$.

\begin{lemma}\label{AX-first-lemma}
For $l \in \Z_+$, we have
$$ 
  \mbox{\rm(a)}\quad \ex \nS{W-nc}^l %\Eq } \sum_{X \in \Z^d} \nS{X-nc}^{l} \int_{[X]} \f_n(t)\,dt  
         \Le  C(l) (nd)^{l/2}, \phantom{XXXXXx}
$$
whenever $n \ge  1/\lmin(\S)$, for universal constants $C(l)$ \adbr{given in Section~\ref{L5.1}.}  
In addition, for each $1\le j\le d$ and $n\ge1$,
\[
   \mbox{\rm(b)}\quad \ex(W_j-nc_j)^2 % \Eq  \sum_{X \in \Z^d} (X_j-nc_j)^2 \int_{[X]} \f_n(t)\,dt 
              \Le \frac12 + 2n\S_{jj},
    \phantom{XXXXX}
\]
and, for $l\in\Z_+$ and for universal constants~$C'(l)$ \adbr{given in Section~\ref{L5.1}},
\eqs
 \phantom{XXX} \mbox{\rm(c)}\quad \ex\{[\S^{-1}(W-nc)]_j^{2l}\} 
     % &=&  \sum_{X \in \Z^d} [\S^{-1}(X-nc)]_j^{2l} \int_{[X]} \f_n(t)\,dt \phantom{XXXX} \\
           &\le& n^l C'(l)(1 + (\S^{-1})_{jj}^l),
\ens
whenever $n \ge d/\{4(\lmin(\S))^2\}$.
\end{lemma}

\msk

The next lemma, proved in Section~\ref{L5.2}, establishes an approximate integration by parts formula
for multivariate discrete normal distributions.
\adbr{We write $I\nuh(X) := I[\nS{X - nc} \le n\h/3]$ for any $\h > 0$, and we say}
that $C \in \KKS$ if~$C$ is an increasing function of $\lmax(\S),1/\lmin(\S)$,
and $C(\d) \in \KKS(\d)$ if $C(\d) \in \KKS$ for each fixed~$\d$. We also define
\eq\label{Oct-ps-sigma-def}
    \smh_\S(n) \Def \frac6{n\sqrt{\lmin(\S)}},
\en
noting that \adbb{its inverse is} $\smh_\S^{-1}(\d) = \smh_\S(\d)$.

\begin{lemma}\label{AX-main-lemma}
Suppose that $W \sim \DN_d(nc,n\S)$.
% and that $f\colon \Z^d\to\re$ satisfies $f(\lfloor nc \rfloor) = 0$, 
% where $\lfloor nc \rfloor := (\lfloor nc_1 \rfloor,\ldots,\lfloor nc_d \rfloor)^T$.  
Then there exist constants
$n_{\ref{AX-main-lemma}}\in\KKS$ and $C\ui_{\ref{AX-main-lemma}}(\d),C\ut_{\ref{AX-main-lemma}}(\d), 
C\uh_{\ref{AX-main-lemma}}(\d)\in \KKS(\d)$, 
such that, for any $n \ge \max\{n_{\ref{AX-main-lemma}},\smh_\S(\d)\}$ 
and for any function $f\colon \Z^d \to \re$, we have 
{\rm 
\begin{enumerate}
 \item $|\ex\{\D f(W)^T b\, I\nud(W)\}
     - n^{-1}\ex\{(f(W)\,(W-nc)^T\S^{-1}b\, I\nud(W)\}| $\\
      {}\hglue0.5in$\Le d^{1/2}C\ui_{\ref{AX-main-lemma}}(\d) n^{-1}\nl{b} \|f\|^\S_{n\d/2,\infty}$;
 \item $|\ex\{\D f(W)^T B(W-nc)\, I\nud(W)\} $\\
    {}\hglue0.1in$\mbox{} - \ex\{f(W)\,[n^{-1}(W-nc)^T \S^{-1}B(W-nc) - \tr B]\, I\nud(W) \}|$\\ 
     {}\hglue0.5in$ \Le d^{1/2}C\ut_{\ref{AX-main-lemma}}(\d)
              n^{-1/2} \|B\|_1\,  \|f\|^\S_{n\d/2,\infty} + \sjd |B_{jj}| \|\D f\|^\S_{n\d/2,\infty}$;\\
 \item $|\ex\{\D f(W)^T B(W-nc)\, I\nud(W)\} $\\
    {}\hglue0.05in$\mbox{} - \ex\{f(W)\,[n^{-1}(W-nc)^T \S^{-1}B(W-nc) - \tr B]\, I\nud(W) \}|$\\ 
     {}\hglue0.2in$ \Le d C\uh_{\ref{AX-main-lemma}}(\d)
              n^{-1/2} \sjd |(\ej)^TB|\,  \|f\|^\S_{n\d/2,\infty} + \sjd |B_{jj}| \|\D f\|^\S_{n\d/2,\infty}$,
\end{enumerate}
}
\nin for any $d$-vector~$b$ and any $d\times d$ matrix~$B$.
The constants $n_{\ref{AX-main-lemma}}$, $C\ui_{\ref{AX-main-lemma}}(\d)$, $C\ut_{\ref{AX-main-lemma}}(\d)$
\adbr{and $C\uh_{\ref{AX-main-lemma}}(\d)$ are defined in~\Ref{Dec-n5.2-def}, \adbb{\Ref{Dec-C5.2(1)-def}, 
\Ref{Dec-C5.2(2)-def}} and following~\Ref{ADB-new-C-dash-defs}, respectively.} 
\end{lemma}

\adbr{With the help of the lemmas above,} we can now show that, if~$W$ has the discrete normal distribution
$\DN_d(nc,n\S)$, then it satisfies the conditions of Theorem~\ref{ADB-first-approx-thm}, with $\e_1 \le c_1n^{-1/2}$,
 $\max\{\e_{20},\e_{21}\} \le c_2 d^{5/2}n^{-1/2}$ and~$\e_{22}=0$, \adbr{and hence that
the conditions of Theorem~\ref{ADB-first-approx-thm} imply a bound on the error of approximating the
distribution of a random $d$-vector by $\DN_d(nc,n\S)$.}

\begin{theorem}\label{AX-discrete-normal}
For~$\S$ positive definite, suppose that $\s^2$, positive definite, and~$A$ are such
that $A\S + \S A^T + \s^2 = 0$; write $\Lbar := d^{-1}\tr(\s^2)$.  Then, if $W \sim \DN_d(nc,n\S)$, 
for any $n \ge \max\{n_{\ref{AX-main-lemma}},\smh_\S(\d)\}$, we have
$$
\begin{array}{rl}
 {\rm (i)}& \ex\nS{W-nc}^2 \Le dC(2) n; \\[1ex]
 {\rm (ii)}& \dtv(\law(W),\law(W+\ej)) \Le  C_{\ref{AX-discrete-normal}}\ui n^{-1/2}, 
             \mbox{ for each } 1\le j\le d; \\[1ex]
 {\rm (iii)}&  |\ex\{\ABA_n h(W)I[\nS{W-nc} \le n\d/3]\}|   \\[1ex]
   &\quad   \Le d^{5/2}n^{-1/2} \Lbar C_{\ref{AX-discrete-normal}}\ut(\d) 
           (\|h\|^\S_{n\d/2,\infty} + n^{1/2}\|\D h\|^\S_{n\d/2,\infty}),
\end{array}
$$
\nin where $\ABA_n$ is as defined in~\Ref{ADB-approx-gen},  $C(2)$ is as in Lemma~\ref{AX-first-lemma}, 
and $C_{\ref{AX-discrete-normal}}\ui$ and~$C_{\ref{AX-discrete-normal}}\ut(\d)$
are \adbb{continuous
functions of $\|A\|/\Lbar$, $\Sp'(\s^2/\Lbar)$ and $\Sp'(\S)$}; %the elements of $\Sp'(\S)$ and~$\Sp'(\s^2/\Lbar)$;
$C_{\ref{AX-discrete-normal}}\ui$ is given in~\Ref{Dec-C5.3(1)-def}, $C_{\ref{AX-discrete-normal}}\ut(\d)$  
implicitly in~\Ref{Dec-C5.3(2)-def}.  \adbb{Hence a random $d$-vector satisfying the conditions
of Theorem~\ref{ADB-first-approx-thm} with $n \ge n_{\ref{AX-discrete-normal}}$ and $0 < \d < \half\td_0(A,\s^2)$
has
\[
    \dtv(\law(W),\DN_d(nc,n\S)) \Le  C_{\ref{AX-discrete-normal}}(v,\d)
           (d^4 (n^{-1/2} + \e_1) +\e_{20} + d^{1/4}\e_{21} +d^{1/2}\e_{22})\log n ,
\]
with 
\eqs
   C_{\ref{AX-discrete-normal}}(v,\d) &:=&  
       C_{\ref{ADB-first-approx-thm}}(v,\d) + C_{\ref{ADB-first-approx-thm}}(C(2),\d)
           (1 + C_{\ref{AX-discrete-normal}}\ui + C_{\ref{AX-discrete-normal}}\ut(\d));\\
   n_{\ref{AX-discrete-normal}} &:=&
         \max\{n_{\ref{ADB-first-approx-thm}}(v,\d),  n_{\ref{AX-main-lemma}},\smh_\S(\d) \}.
\ens
}
\end{theorem}

\begin{proof}
Part~(i) is immediate from~\Ref{AX-DN-moments}, with $v = C(2)$. For Part~(ii), we pick $\d = 1$, and 
then take $b = \ej$ and any function~$f$
with $\|f\|_\infty \le 1$ in Lemma~\ref{AX-main-lemma}(a).  This gives 
$$
   \ex|\D_j f(W)I_n^1(W)| \Le d^{1/2}C\ui_{\ref{AX-main-lemma}}(1) n^{-1} + n^{-1/2}\sqrt{C'(1)(1 + (\S^{-1})_{jj})}, 
$$
in view of Lemma~\ref{AX-first-lemma}(c). For the remaining part of $|\ex\{\D_j f(W)\}|$,
using $\|\D_j f\|_\infty \le 2$, we have
\[
    \ex|\D_j f(W) I[\nS{W-nc} > n/3]| \Le  18dC(2)/n,
\]
by Chebyshev's inequality and from Part~(i), and the estimate follows because $n \ge d^2$,
with 
\eq\label{Dec-C5.3(1)-def}
 C_{\ref{AX-discrete-normal}}\ui \Def C\ui_{\ref{AX-main-lemma}}(1) + \sqrt{C'(1)(1 + (\S^{-1})_{jj})}
          + 18C(2).
\en

For Part~(iii), we 
% work with $\tih(\cdot) := h(\cdot) - h(\lfloor nc \rfloor)$, which leaves
% the differences of the function unchanged, and 
use Lemma~\ref{AX-main-lemma}(b). This gives
\eqa
   \lefteqn{\bigl|\ex\bigl\{\D h(W)^T A(W-nc)I\nud(W)\bigr\} } \non\\
   &-&  \ex\{h(W)\,[n^{-1}(W-nc)^T \S^{-1}A(W-nc) - \tr A]\,I\nud(W)\} \bigr|\non\\
   &&\quad\mbox{} \Le d^{1/2}C\ut_{\ref{AX-main-lemma}}(\d) 
          n^{-1/2} \|A\|_1  \|h\|^\S_{n\d/2,\infty} + \sjd |A_{jj}| \|\D h\|^\S_{n\d/2,\infty} .\label{1} 
\ena
Then, since
\eqs
   \tr(\s^2\D^2h(W)) \Eq  \sid\sjd \s^2_{ij}\D_j f_i(W) ,
\ens
where $f_i(W) := \D_i h(W)$, it follows from Lemma~\ref{AX-main-lemma}(a), 
with $f = f_i$ and with~$b$ the $i$-th column of~$\s^2$, that
\eqa
  \lefteqn{\Bigl|n\ex\{\tr(\s^2\D^2h(W))I\nud(W)\} }\non \\
   &&\mbox{}\quad - \ex\Bigl\{ \sid\sjd \s^2_{ij} \D_ih(W) \{\S^{-1}(W-nc)\}_j I\nud(W) \Bigr\} \Bigr|
   \non\\  &&\Le d^{1/2}C\ui_{\ref{AX-main-lemma}}(\d) \|\s^2\|_1      
                \|\D h\|^\S_{n\d/2,\infty};  \label{2}
\ena
note also that
\eqa 
  \lefteqn{\ex\Bigl\{ \sid\sjd \s^2_{ij} \D_ih(W) \{\S^{-1}(W-nc)\}_j I\nud(W) \Bigr\}} \non\\
   &&\Eq \ex\{\D h(W)^T \s^2 \S^{-1}(W-nc)I\nud(W)\}. 
\ena
But now, from Lemma~\ref{AX-main-lemma}(c), 
\eqa
  \lefteqn{\bigl|\ex\{\D h(W)^T \s^2 \S^{-1}(W-nc)I\nud(W)\} } \non\\
   &&\mbox{}\ - \ex\{h(W)\,[n^{-1}(W-nc)^T \S^{-1}\s^2 \S^{-1}(W-nc) - \tr(\s^2\S^{-1})]I\nud(W)\} \bigr|\non\\
        &\le&    d C\uh_{\ref{AX-main-lemma}}(\d)
           n^{-1/2} \sjd |(\ej)^T \S^{-1}\s^2| \|h\|^\S_{n\d/2,\infty} + \sjd \bigl|[\s^2\S^{-1}]_{jj}\bigr|\,
                  \|\D h\|^\S_{n\d/2,\infty} \non\\
   &\le& d C\uh_{\ref{AX-main-lemma}}(\d)
           n^{-1/2} \{\lmin(\S)\}^{-1} \|\s^2\|_1 \|h\|^\S_{n\d/2,\infty} + \|\s^2\S^{-1}\|_1\, \|\D h\|^\S_{n\d/2,\infty}. 
              \label{3}
\ena
Hence, and since 
\eqs
    &&\|A\|_1 \Le d^{3/2}\|A\|; \quad \|\s^2\|_1 \Le d^{3/2}\lmax(\s^2) 
\ens
and
\eqs
    \|\s^2\S^{-1}\|_1 \Le d^{3/2}\|\s^2\S^{-1}\| \Le d^{3/2}\lmax(\s^2)/\lmin(\S),
\ens
it follows from \Ref1, \Ref2 and~\Ref3 that
\eqa
  \lefteqn{\ex \{\ABA_n h(W) I\nud(W)\}} \non\\
   &=& \ex\bigl\{\bigl(\tr\{A(W-nc)\D h(W)^T\}  + \half n\tr\{\s^2\D^2h(W)\}\bigr)I\nud(W)\bigr\}  \non\\
   &=& \ex\Bigl\{h(W)\,\bigl[\half n^{-1}(W-nc)^T(2\S^{-1}A + \S^{-1}\s^2 \S^{-1})(W-nc) \non\\
   &&\quad\mbox{}\qquad\qquad - \tr A
               - \half \tr(\s^2_\S) \bigr]I\nud(W)\Bigr\} + \th, \label{4}
\ena
where
\eqa
   \leqn{ |\th| \Le d^{1/2}C\ut_{\ref{AX-main-lemma}}(\d) 
             n^{-1/2} \|A\|_1  \|h\|^\S_{n\d/2,\infty} } \non\\
   &&\mbox{}+ \|A\|_1 \|\D h\|^\S_{n\d/2,\infty}
     + \half d^{1/2}C\ui_{\ref{AX-main-lemma}}(\d) \|\s^2\|_1      
                \|\D h\|^\S_{n\d/2,\infty} \non \\
    &&\mbox{}\ + \half d C\uh_{\ref{AX-main-lemma}}(\d)
           n^{-1/2} \{\lmin(\S)\}^{-1} \|\s^2\|_1 \|h\|^\S_{n\d/2,\infty} 
                              + \|\s^2\S^{-1}\|_1\, \|\D h\|^\S_{n\d/2,\infty}\} \non \\
    &&\le\ d^{5/2}n^{-1/2}\Lbar 
     C_{\ref{AX-discrete-normal}}\ut(\d)\bigl(\|h\|^\S_{n\d/2,\infty} + n^{1/2}\|\D h\|^\S_{n\d/2,\infty}\bigr),
            \label{Dec-C5.3(2)-def}
\ena
and $C_{\ref{AX-discrete-normal}}\ut(\d)$ is a function of~$\|A\|/\Lbar$ and the elements of $\Sp'(\S)$, 
$\Sp'(\s^2/\Lbar)$.

Finally, for any $y$ and~$B$, we have $y^TBy = y^TB^Ty = \half y^T(B+B^T)y$, so that
\eqs
   \lefteqn{y^T(2\S^{-1}A + \S^{-1}\s^2 \S^{-1})y \Eq y^T(\S^{-1}A + A^T\S^{-1} + \S^{-1}\s^2 \S^{-1})y }\\
     && \Eq y^T\S^{-1}(A\S + \S A^T + \s^2)\S^{-1}y \Eq 0,\phantom{XXXXXXXX}
\ens
from~\Ref{ADB-Sigma-eqn}, and 
\eqs
    \tr(\s^2_\S) 
        &=& -\tr(\S^{-1/2}A\S^{1/2} + \S^{1/2}A^T\S^{-1/2}) \Eq -2\tr A.
\ens
This, with~\Ref4, establishes that
\eqa
  \leqn{ |\ex \{\ABA_n h(W)I\nud(W)\}| }\non\\
       &&\Le d^{5/2} n^{-1/2} C_{\ref{AX-discrete-normal}}\ut(\d) 
          \Bigl\{ \|h\|^\S_{n\d/2,\infty} + n^{1/2}\|\D h\|^\S_{n\d/2,\infty} \Bigr\}, \phantom{XXX}\label{5}
\ena
as required.  \adbr{The final conclusion follows from the triangle inequality.}
\end{proof}

\adbb{
Discrete normal approximation using Theorem~\ref{AX-discrete-normal} involves checking the conditions of 
Theorem~\ref{ADB-first-approx-thm}.  These can be replaced with analogous conditions in which the norm $\nS{\cdot}$
is replaced by the Euclidean norm.  Here, the parameter~$n$ is also chosen to
standardize $d^{-1}\tr(\S)$; we omit the routine proof.
}

\begin{theorem}\label{ADB-DN-approx-thm}
\adbb{
Let $W$ be a % square integrable 
random vector in~$\Z^d$ with mean $\m := \ex W$ and
positive definite covariance matrix $V := \ex\{(W-\m)(W-\m)^T\}$; define 
$n := \lceil d^{-1}\tr V \rceil$, $c := n^{-1}\m$ and $\S := n^{-1}V$.
Let~$A$ be a $d\times d$ Hurwitz matrix
such that $\s^2 := -(A\S + \S A^T)$ is positive definite, and write $\Lbar := d^{-1}\tr \s^2$.
Set
\[
    \h_0 \Def \frac16\min\Blb \sqrt{\lmin(\S)}, \frac{\lmin(\s^2)}{24\|A\|\sqrt{\Rh(\S)}}\Brb 
                \Eq \tfrac16 \td_0 \sqrt{\lmin(\S)}.
\]
Then, for any $0 < \h \le \h_0$, there exist continuous functions
$C_{\ref{ADB-DN-approx-thm}}(\h), n_{\ref{ADB-DN-approx-thm}}(\h)$ 
of $\|A\|/\Lbar$, $\Sp'(\s^2/\Lbar)$, $\Sp'(\S)$ and~$\h$,
not depending on $d$ or~$n$, with the following property:
if, for some  $\e_1$, $\e_{20}$, $\e_{21}$ and $\e_{22}$, 
and  for some $n \ge n_{\ref{ADB-DN-approx-thm}}(\h)$,
{\rm 
\begin{enumerate}
% \item $\ex|W -nc|^2 \Le dvn$;
 \item $\dtv(\law(W),\law(W+\ej)) \Le \e_1$, for each $1\le j\le d$;
 \item $ |\ex\{\ABA_n h(W)\}I[|W-nc| \le n\h]| \\[1ex]
     \mbox{}\qquad \Le \Lbar (\e_{20}\|h\|_{3n\h_0/2,\infty} + \e_{21}n^{1/2}\|\D h\|_{3n\h_0/2,\infty} 
       + \e_{22}n\|\D^2 h\|_{3n\h_0/2,\infty})$,
\end{enumerate}
}
\nin for all $h\colon \Z^d\to\re$, then it follows that 
\eqs
    \leqn{\dtv(\law(W),\DN_{d}(nc,n\S))} \\
    &&\quad \Le C_{\ref{ADB-DN-approx-thm}}(\h)(d^3n^{-1/2}+ \adbr{d^4}\e_1 
       + \e_{20} + d^{1/4}\e_{21} + d^{1/2}\e_{22})\log n .
\ens
}
\end{theorem}

\ignore{
\begin{theorem}\label{ADB-DN-approx-thm}
\adbr{With the notation and definitions of Theorem~\ref{AX-discrete-normal},
write
\[
    \h_0 \Def \min\Blb \sqrt{\lmin(\S)}, \frac{\lmin(\s^2)}{24\|A\|\sqrt{\Rh(\S)}}\Brb 
                \Eq \td_0 \sqrt{\lmin(\S)}.
\]
}  Then, for any $0 < \h \le \h_0$ and any $v > 0$, there exist 
$C_{\ref{ADB-DN-approx-thm}}(v,\h),n_{\ref{ADB-DN-approx-thm}}(v,\h) < \infty$ 
depending only on $A/\Lbar$, $\s^2/\Lbar$, $v$ and~$\h$, but not on $d$ or~$n$,
with the following property:
if $W$ is any random vector in $\Z^d$ such that, for some  $\e_1$, $\e_{20}$, $\e_{21}$ and $\e_{22}$, 
and  for some $n \ge n_{\ref{ADB-DN-approx-thm}}(v,\h)$,
{\rm 
\begin{enumerate}
 \item $\ex|W -nc|^2 \Le dvn$;
 \item $\dtv(\law(W),\law(W+\ej)) \Le \e_1$, for each $1\le j\le d$;
 \item $ |\ex\{\ABA_n h(W)\}I[|W-nc| \le n\h/6]| \\[1ex]
     \mbox{}\qquad \Le \Lbar (\e_{20}\|h\|_{n\h_0/4,\infty} + \e_{21}n^{1/2}\|\D h\|_{n\h_0/4,\infty} 
       + \e_{22}n\|\D^2 h\|_{n\h_0/4,\infty})$,
\end{enumerate}
}
\nin for all $h\colon \Z^d\to\re$, then it follows that 
\eqs
    \leqn{\dtv(\law(W),\DN_{d}(nc,n\S))} \\
    &&\quad \Le C_{\ref{ADB-DN-approx-thm}}(v,\h)(d^3n^{-1/2}+ \adbr{d^4}\e_1 
       + \e_{20} + d^{1/4}\e_{21} + d^{1/2}\e_{22})\log n .
\ens
\end{theorem}
}

\nin The estimate required in Condition~(b), apart from the truncation to $|W-nc| \le n\h/6$,
is typical of those that are needed for multivariate normal approximation using Stein's
method.  The extra work needed, to translate
multivariate normal approximation into discrete normal approximation in total variation, 
\adbr{lies in establishing Condition~(a) with a suitably small~$\e_1$.  Since, from
Theorem~\ref{AX-discrete-normal},  Condition~(a) is satisfied
with $\e_1 = O(n^{-1/2})$ if $W \sim \DN_d(nc,n\S)$ and~$\S$ is non-singular,
the triangle inequality for a general~$W$ yields
$$
  \dtv(\law(W),\law(W+\ej)) \Le 2\dtv(\law(W),\DN_d(nc,n\S)) + O(n^{-1/2}),
$$
so that~$\dtv(\law(W),\law(W+\ej))$ has to be small if total variation approximation of~$\law(W)$ by the discrete normal 
is to be accurate.}

\ignore{
In applications of the theorem,  the `large' parameter~$n$
needs to be specified.  Condition~(a) of Theorem~\ref{ADB-DN-approx-thm} suggests
that it should be chosen to be comparable to $d^{-1}\ex|W-\ex W|^2$, as is also implied by the use
of $n\S$ as covariance in the approximating discrete normal distribution. This matter
is discussed further in Remark~\ref{Oct-standard}, in the context of linear regression pairs.  
}

We make some effort to make explicit the typical dependence of the error bounds on the dimension~$d$.
This is largely for comparison with the error bounds derived by Bentkus~(2003) and Fang~(2014) for
approximation, with respect to the convex sets metric,
of standardized sums of independent random vectors by the standard $d$-dimensional
normal distribution.  Here, since multiplicative standardization makes no sense in the domain
of random vectors with integer coordinates, there are more quantities than just dimension that
may affect the sizes of the approximation errors.  Nonetheless, we attempt some comparison with the 
above approximations.
To do so, we think of many quantities, such as the eigenvalues of $\s^2$, $A$ and~$\S$, as
being bounded away from zero and infinity as~$d$ varies, and the traces of these matrices
thus being thought of as having 
order~$d$.  This is because, in the standardized setting, using the Stein approach as
in G\"otze~(\adbb{1991}) or Fang~(2014), one has $\s^2 = 2I$, $A = -I$ and $\S = I$.
Our bounds then also involve the values of other parameters, in particular~$\|A\|$ and the elements
of $\Sp'(\s^2)$ and $\Sp'(\S)$, in a way that can be deduced from our arguments, but that we do not attempt to
make explicit, other than that their dependence on these parameters is continuous.
However, we always work in terms of approximations for fixed values of~$n$ and
the parameters of a problem, so that implicit orders of magnitude play no direct part in
the results that we obtain.

\section{Linear regression pairs}\label{regression-property}
\setcounter{equation}0
In this section, we establish a discrete normal approximation theorem for the distribution of
a random vector~$W$, when a copy~$W'$ can be defined on the same probability space, in such
a way that $\ex\{W' \giv W\}$ is approximately a linear function of~$W$. There are many examples 
where this is the case, including those given in Rinott \& Rotar~(1996) and Reinert \& R\"ollin~(2009).
% The theorem is easier to apply if the pair $(W,W')$ is exchangeable, as discussed in
% Section~\ref{exchangeable} below.

Suppose, then, that~$(W,W')$ is a pair of random integer valued $d$-vectors, defined
on the same probability space and having the same distribution.  
Assume that $\ex\{|W|^3\} < \infty$, and write $\m := \ex W$.
Let~$\xi$ denote the difference $W'-W$, so that $\ex\xi = 0$, and set $\s^2 := \ex\{\xi\xi^T\}$,
assumed positive definite.  Suppose that~$\xi$ exhibits an almost linear regression on~$W$, and
that the conditional variance $\s^2(W) := \ex\{\xi\xi^T \giv W\}$ is more or less constant as
a function of~$W$.  Specifically, assume that,
for some $n>0$ and for some $d\times d$ Hurwitz matrix~$A$ with spectral norm~$\moda$, 
we have
\eq\label{Exch-props}
   \begin{array}{rl}
   \ex\{\xi \giv W\} &\Eq n^{-1}A(W-\m) + n^{-1/2}\moda^{1/2}R_1(W); \\[1ex]
   \s^2(W) &\Def \ex\{\xi\xi^T \giv W\} \Eq \s^2 + R_2(W),
   \end{array}
\en 
where~$\ex|R_1(W)|$ and $\ex\|R_2(W)\|_1$ are to be thought of as small.
These two quantities appear explicitly in the bound on 
the error in our discrete normal approximation, and,
clearly, $\ex\{R_1(W)\} = 0$ and $\ex\{R_2(W)\} = 0$.  Let~$\S$ be the positive definite
solution to $A\S + \S A^T + \s^2 = 0$.

\begin{remark}\label{Oct-standard}
{\rm
Note that, in~\Ref{Exch-props}, multiplying $n$ and~$A$ by the same positive constant~$c$
does not change the regression, but~$\S$ is divided by~$c$.  This leaves both $n\S$,
the asymptotic approximation to~$\var W$, and~$\moda/n$
unchanged, the latter implying that~$R_1(W)$ remains the same also.  
The effective data for the problem are the distributions of $\xi$ and~$W$, and in particular~$\s^2$
and $\var W$, and also $\hA := A/n$, which is typically `small'. In order to circumvent the indeterminacy, 
one can compute~$\hS := n\S$, typically `large', by solving 
$\hA\hS + \hS \hA^T + \s^2 = 0$.  Then $\tn := n/\moda$,
$\wA := \tn \hA = A/\moda$ and $\wS := \hS/\tn$ are the same for all~$c$, yield the same regression 
matrix~$\wA /\tilde n = \hA$, and can be used as a standard verson, if required.
}
\end{remark}

We now define further parameters
\eqa
   \adbr{\lla_1} &:=& \adbr{\half\lmin(\S);} \quad \adbr{\vnew} \Def \tr(\s^2_\S)/(d\lla_1); \non \\
    \adbr{\chi} &:=& \adbr{\ex\{|\xi|^3\};}\quad L \Def (\moda/n)^{1/2}\chi \{\tr(\s^2)\}^{-3/2}; \label{ADB-exch-defs-2} \\
       \chi_\S &:=& \ex|\S^{-1/2}\xi|^3;
               \quad L_\S \Def (\moda/n)^{1/2}\chi_\S \{\tr(\s^2_\S)\}^{-3/2} \Le L\Rh(\S)^{3/2}, 
                      \non
\ena
and set $Z := z(W)$, where $z(w) := (nd\vnew)^{-1/2}\S^{-1/2}(w-\m)$.  $L$, $L_\S$ and~$Z$ all
involve $A$, $n$ and~$\S$ only through the standardized quantities $n/\moda$ and $n\S$.  We then
assume that
%, for $\ps_{11}, \ps_{12}$ and~$\ps_2$ such that
% \eq\label{Exch-ps-cond}
%   \ps_{11} \le 1/2,\quad \ps_{12} \Le 1/4 \quad\mbox{and}\quad \ps_2 \Le 1/4,
% \en
the following inequalities hold:
\eqa
%   \begin{array}{rl}
    \{\moda/\lla_1\}^{1/2} \ex\{(1+|Z|)\,|\S^{-1/2}R_1(W)|\} &\le& \half(\tr(\s^2_\S))^{1/2}(1 + \ex|Z|^2);
           \label{Exch-R-bnds-11}\\%[1ex]
    \{\moda/\lla_1\}^{1/2} \ex\{|Z|(1 + |Z|)|\S^{-1/2}R_1(W)|\} &\le& \quarter(\tr(\s^2_\S))^{1/2}(1 + \ex|Z|^3). 
      \phantom{XXX}
            \label{Exch-R-bnds-12}
%   \ex\{(1 + |Z|)\|\S^{-1/2}R_2(W)\S^{-1/2}\|_1 \} &\le& \quarter\tr(\s^2_\S)(1 + \ex|Z|^3).
%                \label{Exch-R-bnds-2}
%   \end{array}
\ena
They can reasonably be expected to be satisfied if $|R_1(W)|$ is indeed small.
In particular, \Ref{Exch-R-bnds-11}--\Ref{Exch-R-bnds-12} are satisfied if
\eqa
% \{\ex|R_1(W)|^2\}^{1/2} &\le& (\tr(\s^2))^{1/2};\non\\ 
             \{\ex|\S^{-1/2}R_1(W)|^3\}^{1/3} &\le&  \eighth(\lla_1\tr(\s^2_\S)/\moda)^{1/2}. 
                 \label{Exch-simpler-bnds}
%  \{\ex\|\S^{-1/2}R_2(W)\S^{-1/2}\|_1^{3/2}\}^{2/3} &\le& \eighth\tr(\s^2_\S). \label{Exch-simpler-bnds}
\ena

Under the above conditions, the second and third moments of
$|Z|$ can be suitably bounded; the proof is given in Section~\ref{Exch-moments-appx}.

\begin{lemma}\label{Exch-moments}
If \Ref{Exch-R-bnds-11} and~\Ref{Exch-R-bnds-12} hold,
% are valid for $\ps_{11},\ps_{12}$ and~$\ps_2$ satisfying~\Ref{Exch-ps-cond}, 
and if $n/\lla_1 \ge 1$, then 
\[
    \ex|Z|^2 \Le 2;\quad \ex|Z|^3 \Le m_3 \Def 2\Bigl(1 + \frac{10\chi_\S}{(\tr(\s^2_\S))^{3/2}}\Bigr),
\]
where $Z = z(W)$, with $z(w)$ as defined above.  In particular, for any $\d > 0$,
\[
    \frac n{\moda} \pr[\nS{W-\m} > n\d\moda^{-1/2}] 
        \Le 2d^{3/2}\d^{-3}((\moda/n)^{1/2} + 10 L_\S) \Blb \frac{2\lbar(\s^2_\S)}{\lmin(\s^2_\S)}\Brb^{3/2}.
\]
\end{lemma}

\begin{remark}
{\rm
 Note that
\eq\label{Oct-standard-event}
    \Blb \nS{W-\m} > \frac{n\d}{\sqrt{\moda}} \Brb \Eq 
          \Blb |Z| > \d\sqrt{\frac n{\moda}}\,\sqrt{\frac{\lmin(\s^2_\S)}{2\lbar(\s^2_\S)}}\Brb
\en
involves only standardized quantities.
}
\end{remark}

We are now in a position to prove a discrete normal approximation theorem.  To state it,
we introduce some further notation:
\eq\label{Exch-dtv-defs}
 \begin{array}{rl}
  \e_1 &:=\ \max_{1\le j\le d}\dtv(\law(W),\law(W+\ej)); \\[1ex]
   \e_1(\xi) &:=\ \max_{1\le j\le d}\dtv(\law(W \giv \xi),\law(W+\ej \giv \xi)).
 \end{array} 
\en

\begin{theorem}\label{AX-exch-thm} 
Assume that~$(W,W')$ is a pair of random integer valued $d$-vectors, such that $\law(W) = \law(W')$
and that $\ex|W|^3 < \infty$; write $\m := \ex W$.
Suppose that $\xi := W'-W$ satisfies the regression condition~\Ref{Exch-props}, for matrices $A$ and~$\s^2$ such 
that~$A$is Hurwitz and~$\s^2$ is positive definite; let~$\S$ 
be the positive definite solution of $A\S + \S A^T + \s^2 = 0$. 
% Assume also that~$|\xi|$ has finite third moment~$\chi$.
Define $\ex|\xi|^3 := \chi$, %$\lla_1 := \half\lmin(\s^2_\S)$, 
$\Lbar := d^{-1}\tr(\s^2)$ and 
$L := (\moda/n)^{1/2}\chi\{\tr(\s^2)\}^{-3/2}$,
and assume that \Ref{Exch-R-bnds-11} and~\Ref{Exch-R-bnds-12} hold.  
Let $\wA$ and~$\wS$ be as in Remark~\ref{Oct-standard}.
Then there exist constants~$n_0$ and~$C$, depending on $\|\wA\|$ and $\s^2$, 
such that, if $n/\moda \ge n_0$, we have 
\eqs
   \leqn{ d_{TV}\bigl(\law(W),\DN_{d}(\mu,n\S)\bigr) }\\
  &&\Le  C \log n\bigl\{d^{3}(\moda/n)^{1/2} + \adbr{d^4}\e_1  + d^{1/4}\ex|R_1(W)|  \\
   &&\qquad\qquad\mbox{}\hskip1in     + d^{1/2}\ex\|R_2(W)\|_1  + d^3L + d^2\ex\{|\xi|^3 \e_1(\xi)\} \bigr\}.
\ens
\end{theorem}

\begin{proof} 
Because $\law(W) = \law(W')$, we have
\eqa
  0 &=& (n/\moda)\ex\{h(W')I[\nS{W'-\m} \le M] - h(W)I[\nS{W-\m} \le M]\} \non\\
    &=& (n/\moda)\ex\{(h(W')-h(W))I[\nS{W-\m} \le M]\} \non\\
    &&\mbox{} + (n/\moda)\ex\{h(W')(I[\nS{W'-\m} \le M] - I[\nS{W-\m} \le M])\},\phantom{XX}\label{Aug-Exch-0}
\ena
for any \adbr{function $h\colon \Z^d \to \re$ and} $M > 0$. We shall take $M = n\h/6\sqrt{\moda}$, for~$\h$ to be 
prescribed later, in view of~\Ref{Oct-standard-event}.
For bounded functions $h$, the second term can be simply estimated, 
using Lemma~\ref{Exch-moments}, by 
\eqa
  \eth_0 &:=& 2(n/\moda) \|h\|_\infty \pr[\nS{W-\m} > M] \non\\
%           \Le \frac{432 m_3}{\h^3 \sqrt n} \Blb \frac{\tr(\s^2_\S)}{\moda} \Brb^{3/2}\|h\|_\infty,
     &\le&  864\,  d^{3/2}\h^{-3}((\moda/n)^{1/2} + 10L_\S) \Blb \frac{2\lbar(\s^2_\S)}{\lmin(\s^2_\S)}\Brb^{3/2} 
                  \|h\|_\infty .  \label{ADB-Exch-0}
\ena
For the first term, we write
\eq\label{Aug-Exch-1}
   h(W')-h(W) \Eq \xi^T\D h(W) + \half \xi^T \D^2 h(W) \xi + e_2(W,\xi,h),
\en
\adbr{thus defining $e_2(X,J,h)$.}
From~\Ref{Exch-props}, its first element yields
\eqa
  \leqn{ \frac n{\moda} \bigl|\ex\{\xi^T\D h(W)I[\nS{W-\m} \le M]\}} \non\\
  && \mbox{}\qquad -\ \ex\{n^{-1}(W-\m)^TA^T\D h(W)I[\nS{W-\m} \le M]\}\bigr| \non \\
  &&  \Le (n/\moda)^{1/2} \ex\{|R_1(W)^T \D h(W)|I[\nS{W-\m} \le M]\} \non  \\
  && \Le  (n/\moda)^{1/2} \ex|R_1(W)| \|\D h\|^\S_{\frac{n\h}{6\sqrt{\moda}},\infty}\ =:\ \eth_1. \label{ADB-Exch-1}
\ena
Then
\eqa
  \leqn{ \frac n{2\moda} \bigl|\ex\{\xi^T \D^2 h(W) \xi\, I[\nS{W-\m} \le M]\} }\non\\
   &&\qquad\qquad\mbox{}   - \ex\{\tr(\s^2\D^2h(W))I[\nS{W-\m} \le M]\} \bigr| \non\\ 
   &&\Le \half \ex\{\|R_2(W)\|_1\}\,(n/\moda) \|\D^2 h\|^\S_{\frac{n\h}{6\sqrt{\moda}},\infty} \ =:\ \eth_2.
        \phantom{XXXXXXX} \label{ADB-Exch-2}
\ena

It remains to bound $(n/\moda) \ex\{e_2(W,\xi,h) I[\nS{W-\m} \le M]\}$.  
We first consider $|\xi| > \sqrt {n/\moda}$, and use the bound
\eq\label{Aug-xx}
    \ex\{\nl{\xi}^r\,I[|\xi| > \sqrt{n/\moda} ]\} \Le d^{r/2}\ex\{|\xi|^r\,I[|\xi| > \sqrt{n/\moda}]\}
           \Le d^{r/2}\chi (n/\moda)^{-(3-r)/2}
\en
for $r = 0,1, 2$. Since 
\eqs
    \leqn{ |e_2(W,\xi,h)|\,I[\nS{W-\m} \le M] }\\
    &&\Le 2\|h\|_\infty + \nl{\xi} \|\D h\|^\S_{\frac{n\h}{6\sqrt{\moda}},\infty}
        + \half \nl{\xi}^2  \|\D^2 h\|^\S_{\frac{n\h}{6\sqrt{\moda}},\infty},
\ens
it follows, using~\Ref{Aug-xx}, that
\eqa
    \eth_3 &:=& \frac n{\moda} \ex\{|e_2(W,\xi,h)|\,I[\nS{W-\m} \le M]\,I[|\xi| > \sqrt{n/\moda}] \} 
                   \label{ADB-Exch-3}\\
     &\le& \chi\sqrt{\frac{\moda}n}\Bigl\{ 2\|h\|_\infty  
          + \Bigl(\frac{dn}{\moda}\Bigr)^{1/2} \|\D h\|^\S_{\frac{n\h}{6\sqrt{\moda}},\infty}
          + \frac{dn}{2\moda}\, \|\D^2 h\|^\S_{\frac{n\h}{6\sqrt{\moda}},\infty} \Bigr\} \non\\
     &\le& 2L\{\tr(\s^2)\}^{3/2} \Bigl\{ \|h\|_\infty  
        + \Bigl(\frac{dn}{\moda}\Bigr)^{1/2} \|\D h\|^\S_{\frac{n\h}{6\sqrt{\moda}},\infty}
        +  \frac{dn}{\moda} \|\D^2 h\|^\S_{\frac{n\h}{6\sqrt{\moda}},\infty} \Bigr\}.\non
\ena

For $|\xi| \le \sqrt{n/\moda}$, \adbr{we split $e_2(W,\xi,h)$ into a sum of third differences and a remainder:}
\eq\label{Exch-3rd-bit}
   e_2(W,\xi,h) \Eq E_2(W,\xi,h) - \half \sjd \xi_j\D_{jj} h(W).
\en
For the contribution from the second term in~\Ref{Exch-3rd-bit}, we have at most
\eqa
 \eth_4 &:=& \frac n{2\moda} \Bigl|\ex\Bigl\{ \sjd \xi_j\D_{jj} h(W)
                     I[\nS{W-\m} \le M]I[|\xi| \le \sqrt{n/\moda}]\Bigr\}\Bigr|\non\\
        &\le& \frac n{2\moda} \Bigl|\ex\Bigl\{ \sjd \xi_j\D_{jj} h(W)I[\nS{W-\m} \le M]\Bigr\}\Bigr| \non\\
        &&\qquad\mbox{} + \frac12 \ex\{\nl{\xi} \,I[|\xi| > \sqrt{n/\moda}]\}\,\frac n{\moda}
                     \|\D^2 h\|^\S_{\frac{n\h}{6\sqrt{\moda}},\infty}\non\\
        &=:& \eth_4' + \eth_4'',\label{ADB-Exch-4-}
\ena
say. Now, \adbb{recalling} \adbr{$\vnew := \tr(\s^2_\S)/(d\lla_1)$ 
and~$Z :=  (nd\vnew)^{-1/2}\S^{-1/2}(W-\m)$,}
\Ref{Exch-props} and~\Ref{Exch-R-bnds-11} give
\eqa
    \eth_4' 
     &=& \frac1{2\moda} \sjd \Bigl| \ex\bigl\{([A(W-\m)]_j 
                       + n^{1/2}\moda^{1/2}[R_1(W)]_j)\,\D_{jj}h(W) I[\nS{W-\m} \le M]\bigr\}\Bigr| \non\\
     &\le& \half n^{-1/2} \Blb  (d\vnew)^{1/2}  \ex\nl{A\S^{1/2}Z}  +  \moda^{1/2}\ex\nl{R_1(W)} \Brb 
                 \,\frac n{\moda} \|\D^2 h\|^\S_{\frac{n\h}{6\sqrt{\moda}},\infty} \non\\[1ex]
     &\le& \half (\moda/n)^{1/2} d \sqrt \vnew(\moda^{-1/2}\|\S^{1/2}\|)
                  \{\sqrt2\|A\| + \tfrac32 \lla_1\} \,(n/\moda) \|\D^2 h\|^\S_{\frac{n\h}{6\sqrt{\moda}},\infty}.
               \label{ADB-Exch-4}
\ena
Then, from~\Ref{Aug-xx},
\eq\label{ADB-eta-4''}
  \eth_4'' \Le \half \{d^{1/2} (\moda/n)\chi\} \,(n/\moda)\|\D^2 h\|^\S_{\frac{n\h}{6\sqrt{\moda}},\infty}.
\en

For the first term in~\Ref{Exch-3rd-bit}, we use Lemma~4.4(i) and Remark~4.5 of Part~I to conclude that, 
if $|\xi| \le \sqrt{ n/\moda}$ and $n\h/24\sqrt{\moda} \ge \sqrt{n/\moda\lmin(\S)}$, then
\eqs
   \leqn{\eth_5(\xi) \Def (n/\moda)|\ex\{ E_2(W,\xi,h)I[\nS{W-\m} \le M] \giv \xi\}| } \\ 
    &\le& \bigl\{d^{3/2}|\xi|^3 \e_1(\xi) + 2d|\xi|^2 \pr[\nS{W-\m} \ge M/4 \giv \xi]\bigr\}\, 
                     (n/\moda)\|\D^2 h\|^\S_{\frac{n\h}{4\sqrt{\moda}},\infty}.
\ens
Taking expectations, and then using Lemma~\ref{Exch-moments}, this gives
\eqa
    \eth_5 &:=& \ex\{|\eth_5(\xi)|I[|\xi| \le \sqrt{n/\moda}]\} \label{ADB-Exch-5} \\
           &\le& \Blb d^{3/2}\ex\{|\xi|^3\e_1(\xi)\} 
            + \frac {2dn}{\moda}\pr[\nS{W-\m} \ge M/4]\Brb \, 
                 \frac n{\moda}\|\D^2 h\|^\S_{\frac{n\h}{4\sqrt{\moda}},\infty}
                           \non \\
 %          &\le& \{d^{3/2}\ex\{|\xi|^3\e_1(\xi)\} 
 %                  + 2d(24/\h)^3 n^{-2}\ex\{\nS{W-\m}^3\}\}\,n\|\D^2 h\|^\S_{n\h/4,\infty} \\
           &\le& \Bigl\{d^{3/2}\ex\{|\xi|^3\e_1(\xi)\}
                    + \h^{-3}d^{5/2}C \{\Rh(\s^2_\S)\}^{3/2}\Bigl(\sqrt{\frac {\moda}n} + L_\S\Bigr) \Bigr\}\,
                 \frac n{\moda}\|\D^2 h\|^\S_{\frac{n\h}{4\sqrt{\moda}},\infty}, \non
\ena
for $C$ a universal constant.

Let
\eq\label{Oct-tilde-gen}
    {\widetilde\AA}_{\tn} h(w) \Def \half \tn\tr(\s^2\D^2h(w)) + (w-\m)^T\wA^T\D h(w).
\en
Then,
combining the estimates \Ref{ADB-Exch-0} and \Ref{ADB-Exch-1}--\Ref{ADB-Exch-5} 
with \Ref{Aug-Exch-0} and~\Ref{Aug-Exch-1}, 
% and using Lemmas \ref{Exch-moments} and~\ref{Exch-dtv-lemma}, 
we have shown that 
\eqa
  \leqn{ |\ex\{{\widetilde\AA}_{\tn} h(W) I[|W-\m|_{\wS} \le \tn\h/6]\}| } \label{Aug-line-1}\\
    &&\Eq \bigl|\half \tn \ex\{\tr(\s^2\D^2h(W))I[\nS{W-\m} \le n\h/6\sqrt{\moda}]\}   \non\\
    &&\qquad\quad\mbox{}    + \ex\{(W-\m)^T \wA^T\D h(W)I[\nS{W-\m} \le n\h/6\sqrt{\moda}]\} 
             \bigr|\phantom{XXX}  \non\\
    &&\Le \sum_{l=0}^3 \eth_l + \eth'_4 + \eth''_4 + \eth_5 \non\\
    &&\Le \e_{20}\|h\|_\infty + \e_{21} \tn^{1/2}\|\D h\|^\S_{\frac{n\h}{4\sqrt{\moda}},\infty} 
                    + \e_{22} \tn \|\D^2 h\|^\S_{\frac{n\h}{4\sqrt{\moda}},\infty} \non\\
    &&\Le \Lbar\{\e_{20}'\|h\|_\infty + \e_{21}' \tn^{1/2}\|\D h\|^\wS_{n\h/4,\infty} 
                      + \e_{22}' \tn\|\D^2 h\|^\wS_{n\h/4,\infty}\},
            \label{Aug-line-2}
\ena
with 
\eqs
  \e_{20}' &=& C_0(\h)d^{3/2}(\tn^{-1/2}+L); \qquad 
                       \e_{21}' \Eq \Lbar^{-1}(\ex|R_1(W)| + 2Ld^2\Lbar^{3/2}); \\
    \e_{22}' &=& C_2(\h)(\ex\|R_2(W)\|_1 + Ld^{5/2} + d^{5/2}\tn^{-1/2} + d^{3/2}\ex\{|\xi|^3\e_1(\xi)\}),
\ens
where the constants $C_l(\h)$ depend on $\h$, $\|\wA\|$ and the elements of $\Sp'(\s^2)$ and $\Sp'(\wS)$.
Since, if $\tn\h/12 > 2\{\lmin(\wS)\}^{-1/2}$, the quantity in~\Ref{Aug-line-1} does not change if~$h(X)$ is replaced 
by zero for $|X-\m|_{\wS} > \tn\h/4$,
the norm $\|h\|_\infty$ can be replaced by $\| h\|^{\wS}_{\tn\h/4,\infty}$ for such $\tn$ and~$\h$.
Thus Condition~(iii) of Theorem~\ref{ADB-first-approx-thm} is satisfied, 
for~$\widetilde\AA_n$ as defined in~\Ref{Oct-tilde-gen}, if we take $\h = \td_0$, for~$\td_0$ as defined in
Theorem~\ref{ADB-first-approx-thm}, and for~$\tn$ such that % $\tn\td_0/12 > 2/\sqrt{\lmin(\wS)}$
$\tn \ge \max\bigl\{n_{\ref{ADB-first-approx-thm}},24/(\td_0\{\lmin(\wS)\}^{1/2})\bigr\}$. 
The remaining conditions of Theorem~\ref{ADB-first-approx-thm}, with $\wS$ for~$\S$ and with $\tn$ for~$n$, are
easily checked: Condition~(i) is implied by Lemma~\ref{Exch-moments}, with \adbr{$v = 2\vnew$}, and Condition~(ii) is 
just~\Ref{Exch-dtv-defs}. This proves the theorem.
\end{proof}

\begin{remark}\label{Exch-sigma-fields}
{\rm 
 Direct computation of the quantities $\ex|R_1(W)|$ and $\ex\|R_2(W)\|_1$ can be awkward.
It may be easier to find bounds on
\[
   \tR_1 \Def n^{1/2}\{ \ex(\xi \giv \FF) - n^{-1}A(W-\m)\}\quad\mbox{and}\quad
   \tR_2 \Def \ex(\xi\xi^T \giv \FF) - \s^2,
\]
for a $\s$-field $\FF$ such that~$W$ is $\FF$-measurable.  From
the properties of conditional expectation and Jensen's inequality, it follows that,
for any non-negative random variable~$Y(W)$, we have
\eqs
   \ex\{Y(W)|R_1(W)|\} &\le& \ex\{Y(W)|\tR_1|\} ;\\ %\quad\mbox{and}\quad
     \ex\{Y(W)\|R_2(W)\|_1\} &\le& \ex\{Y(W)\|\tR_2\|_1\}.
\ens
Hence we can use $\tR_1$ and $\tR_2$ in place of $R_1(W)$ and~$R_2(W)$ when computing the bounds
in the theorem and in verifying conditions \Ref{Exch-R-bnds-11}--\Ref{Exch-R-bnds-12}.
}
\end{remark}

\ignore{
In general, it may be difficult to find expressions bounding the total variation distance
$\dtv(\law(W),\law(W+\ej))$, and its conditional versions.  There are two circumstances in which
this can be satisfactorily accomplished; when~$W$ contains an embedded sum of independent random variables,
and when the pair $(W,W')$ is {\it exchangeable\/}.   These are discussed in the following sections. 
}

\section{Examples}\label{examples}
\setcounter{equation}0
\adbr{In Part~I, following the proof of Theorem~5.3, it was remarked that, using Theorem~\ref{AX-discrete-normal},
error bounds of order $O(n^{-1/2}\log n)$ for the (quasi-)equilibrium distributions
of rather general Markov jump processes can be proved.  Here, we concentrate
on examples exhibiting the linear regression structure of the previous section.}

\subsection{Sums of independent integer valued random vectors}\label{indep}
Let $Y_i$, $1\le i\le m$, be independent $\Z^d$-valued random vectors, with 
means~$\m_i$ and covariance matrices~$\VVV_i$, and let $\g_i := \ex|Y_i - \m_i|^3$.
% Without loss of generality, we shall assume that each of the~$Y_i$ has been
% translated in such a way that $\nl{\m_i} \le d/2$ for each~$i$.
Write $\pr[Y_i=X]=:p_{i,X}$, $X\in\Z^d$, and define 
$u_i:=\min_{1\le j\le d}\{1-d_{TV}(\law(Y_i),\law(Y_i+\ej))\}$.
Let 
\eqs
   W &:=& \sum_{i=1}^m Y_i;\quad \m \Def \ex W \Eq \siim \m_i; \quad \vvs_m \Def \siim u_i;\\
   \VVV &:=& \ex\{(W-\m)(W-\m)^T\} \Eq \siim \VVV_i; \quad \Gamma \Def \siim \g_i.
\ens
We apply Theorem~\ref{AX-exch-thm} to approximate the distribution of~$W$.

To start with, we need
to define a~$W'$ on the same probability space, in such a way that $\law(W') = \law(W)$, and
such that $\xi = W'-W$ is not too large.  The canonical way to do this (Stein, 1986, p.16) is
to let $(Y_1',\ldots,Y_m')$ be an independent copy of $(Y_1,\ldots,Y_m)$, and to let~$K$ be uniformly
distributed on $\{1,2,\ldots,m\}$, independently of the $Y_i$ and the~$Y_i'$; then~$W'$ is taken to
be $W - Y_K + Y_K'$.   It is clear that $\law(W') = \law(W)$, and also, writing 
$\xi := W'-W = Y_K' - Y_K$, that
\eqs
    \ex(\xi \giv W) &=& \ex\{\ex(\xi \giv Y_1,\ldots,Y_m) \giv W\}\\
    &=& \ex\Bigl\{m^{-1}\siim (\m_i - Y_i) \Giv W\Bigr\}
                    \Eq - m^{-1}(W - \m),
\ens
so that the regression condition in~\Ref{Exch-props} is satisfied with $A/n = -I/m$,
and with $R_1(W) = 0$.  Then $\s^2 = \ex\{\xi\xi^T\} = 2\VVV/m$, giving, for the
standardized quantities of Remark~\ref{Oct-standard}, $\S_0 = \VVV$ and $\lla_0 = 1/m$,
and hence $\tn = m$, $\wA = -I$ and $\wS = \VVV/m$.
Note also that
\[
    \chi \Eq \ex|\xi|^3 \Eq m^{-1}\siim \ex|Y_i - Y_i'|^3 \Le 4m^{-1}\Gamma.
\]

As a next step in applying Theorem~\ref{AX-exch-thm}, we show that 
the quantity~$\e_1$ of~\Ref{Exch-dtv-defs} can be suitably bounded.

\begin{lemma}\label{AX-indep-1} For~$W$ as defined above,
$$
       \e_1 \Def  \max_{1\le j\le d}d_{TV}(\law(W),\law(W+\ej))\Eq O(\vvs_m^{-1/2}).
$$
\end{lemma}

\begin{proof} 
Fix any $1\le j\le d$, and, for $X \in \Z^d$, define 
$$
   p^-_{i,X} \Def \half(p_{i,X}\wedge p_{i,X-\ej});\quad p^+_{i,X} \Def \half(p_{i,X}\wedge p_{i,X+\ej}).
$$
Then define the pair $(Y_i,\tY_i)$ jointly, for $1\le i\le m$, by
$$
  (Y_i,\tY_i) \Eq\left\{
         \begin{array}{ll}(X,X-\ej)&\mbox{with probability }p_{i,X}^-;\\
                          (X,X+\ej)&\mbox{with probability }p_{i,X}^+;\\
                          (X,X)&\mbox{with probability }p_{i,X}-p_{i,X}^- - p_{i,X}^+,
    \end{array}\right.\ X \in \Z^d.
$$
Set $Z_i:=Y_i-\tY_i$. Then $Z_i$ takes the values $\ej$ and $-\ej$ each with probability 
$\sum_{X\in\Z^d} p_{i,X}^+$, and takes the value 0 with probability 
$1-\sum_{X\in\Z^d} p_{i,X}\wedge p_{i,X+\ej}$. Hence, for $T_0 := 0$ and $T_k:=\sum_{i=1}^k Z_i$, 
the process $\{T_k,\,0\le k\le m\}$ is a lazy symmetric random walk. Define 
$$
 Y_i' \Def \left\{\begin{array}{ll}
               \tY_i,&i\le \tau;\\
               Y_i,&i>\tau,\end{array}\right.
$$
where $\tau := \min\{k\colon 1\le k\le m, T_k=\ej\}$ if this is defined, and with $\tau=m$ otherwise.
Set $W'=\sum_{i=1}^m Y_i'$. Then, by the Mineka coupling argument (Lindvall, 2002, Section~II.14),
it follows that
$$
   d_{TV}(\law(W),\law(W+\ej))\Le \pr[W\ne W'+\ej] \Le \pr[\tau>m] \Eq O(\vvs_m^{-1/2}).
$$ 
\end{proof}

As a result of this lemma, it is clear that the quantity~$\e_1$ of~\Ref{Exch-dtv-defs} is
of order $O(\vvs_m^{-1/2})$. 
Defining $W\uii := W - Y_i$ and $\ts_m := \vvs_m - \max_{1\le i\le m}u_i$, we now observe that,
for any $X \in \Z^d$, the conditional quantity $\e_1(X)$ is bounded by
\eq\label{epsilon-tilde}
   \te_1 \Def \max_{1\le i\le m}\max_{1\le j\le d}\dtv(\law(W\uii),\law(W\uii+\ej)) \Eq O((\ts_m)^{-1/2}),
\en
with the final order statement following directly from Lemma~\ref{AX-indep-1}.  This is because,
for any $X \in \Z^d$, 
\eqs
   \leqn{\dtv(\law(W+\ej \giv \xi = X),\law(W \giv \xi = X)) }\non\\
   &&\Le m^{-1}\siim \dtv(\law(W+\ej \giv \xi = X, K=i),\law(W \giv \xi = X, K=i)),
\ens
and because, by independence,
\eqs
    \leqn{\dtv(\law(W+\ej \giv \xi_i = X),\law(W \giv \xi_i = X)) }\\
   &\le&
     \ex\{\dtv(\law(W\uii + \xi_i +\ej \giv \xi_i,\xi_i' = \xi_i + X),
                  \law(W\uii + \xi_i \giv \xi_i,\xi_i' = \xi_i + X))\} \\
     &=& \ex\{\dtv(\law(W\uii + \xi_i +\ej \giv \xi_i),
                  \law(W\uii + \xi_i \giv \xi_i))\} \\
     &=& \ex\{\dtv(\law(W\uii +\ej \giv \xi_i), \law(W\uii \giv \xi_i))\}\\
     &=& \ex\{\dtv(\law(W\uii +\ej), \law(W\uii))\} \Le \te_1.
\ens

Thus a number of the elements appearing in the bound given in Theorem~\ref{AX-exch-thm} can
be successfully handled.  We now show that a multivariate discrete normal approximation can indeed
be established.   We write
$$
   \Lbar \Def d^{-1}\tr(\s^2) \Eq 2\tr(\VVV/m) \quad\mbox{and}\quad
  L \Def m^{-1/2}\frac{\chi}{\{\tr(\s^2)\}^{3/2}} \ \ge\ m^{-1/2};
$$
the latter quantity, introduced in~\Ref{ADB-exch-defs-2}, is of order $O(m^{-1/2})$ 
if the ratio $\ex|\xi|^3/\{\ex|\xi|^2\}^{3/2}$ remains bounded.

\begin{theorem}\label{AX-indep-thm} 
Under the above circumstances,
$$
   d_{TV}\bigl(\law(W),\DN_{d}(\mu,\VVV)\bigr) 
            \Le C d^{7/2} \log m \bigl(L + \adbr{(d/m)^{1/2}}\bigr) \sqrt{\frac{m}{\ts_m}},
$$
for a suitable constant~$C$, depending only on~$\Sp'(\VVV/m)$.
\end{theorem}

\begin{proof}
With the definitions of $W'$ and~$W$ given above, the regression condition in~\Ref{Exch-props}
is satisfied with~$R_1(w) = 0$ for all $w \in \Z^d$, so that 
Conditions \Ref{Exch-R-bnds-11} and~\Ref{Exch-R-bnds-12} are trivially satisfied.  
Then $\e_1 = O(\vvs_m^{-1/2})$, by Lemma~\ref{AX-indep-1},
and 
\eq\label{Sept-xi^3-bnd}
   \ex\{|\xi|^3 \e_1(\xi)\} \Eq O((\ts_m)^{-1/2}\chi),
\en
from the observations above.  Note that
\eq\label{Oct-chi-bnd}
   \chi \Eq L\sqrt m\,d^{3/2}\Lbar^{3/2}.
\en

For $\ex\|R_2(W)\|_1$, for any $X,w \in \Z^d$, we write $p(X) := \pr[\xi=X]$, obtaining
\[
    \s^2_{il}(w) \Eq \sum_{X\in\Z^d} X_i X_l \pr[\xi=X\giv W=w] 
              \Eq \sum_{X\in\Z^d}p(X) X_i X_l \frac{\pr[W=w \giv \xi=X]}{\pr[W=w]}.
\]
Hence
\eqa
  \leqn{\ex|\s^2_{il}(W) - \s^2_{il}| 
     \Eq \sum_{w\in\Z^d}\Bigl|\sum_{X\in\Z^d}p(X) X_i X_l (\pr[W=w \giv \xi=X] - \pr[W=w])\Bigr| }\non\\
     &=& \sum_{w\in\Z^d}\Bigl|\sum_{X\in\Z^d}p(X) X_i X_l \sum_{y\in\Z^d}p(y)
                  (\pr[W=w \giv \xi=X] - \pr[W=w\giv\xi=y])\Bigr|\non \\
     &\le& \sum_{X\in\Z^d}p(X) |X_i| |X_l| \sum_{y\in\Z^d}p(y)\,2\dtv(\law(W \giv \xi=X),\law(W \giv \xi=y)).
            \label{Sept-Ind-1}
\ena
Now, by independence,
\eqs
   \leqn{\dtv(\law(W \giv \xi=X),\law(W \giv \xi=y))}\\
     &&\Le \frac1m \siim \dtv(\law(W\uii + Y_i \giv  Y_i' - Y_i = x),\law(W\uii + Y_i \giv  Y_i' - Y_i = y)) \\
     &&\Le \frac1m \siim \ex\{\dtv(\law(W\uii + Y_i \giv  Y_i',  Y_i = Y_i'-x),\law(W\uii + Y_i \giv  Y_i', Y_i= Y_i'-y))\} \\
     &&\Eq \frac1m \siim \ex\{\dtv(\law(W\uii),\law(W\uii-y+x))\} \\
     && \le \te_1\nl{y-x}.
\ens
Substituting this bound into~\Ref{Sept-Ind-1} and adding over $1\le i,l \le d$ thus gives
\eqa
   \ex\|R_2(W)\|_1 &\le& 2\sid \sld \sum_{X\in\Z^d}p(X) |X_i| |X_l| \sum_{y\in\Z^d}p(y) \te_1\nl{x-y} \non\\
     &\le& 2\te_1\sum_{X\in\Z^d}p(X) \nl{X}^2 \{\nl{X} + \ex \nl{\xi}\} \non\\
     &\le& 4\te_1 \ex\nl{\xi}^3 \Le 4\te_1 d^{3/2}\chi . \label{Sept-R2-bnd}
\ena

It only remains to collect the elements needed for Theorem~\ref{AX-exch-thm}.  
% It is immediate that $\gbars/\Lbar = \gbar(\VVV/m)$. 
%, and that~$\h_0$ is a function of the elements of~$\Sp'(\VVV/m)$ alone.  
\adbr{From \Ref{Sept-xi^3-bnd} and~\Ref{Sept-R2-bnd},} 
and from the definition of~$L$, we have
\[
    d^{1/2}\ex\|R_2(W)\|_1 + d^2 \ex\{|\xi|^3 \e_1(\xi)\}
      \Eq O(d^2\chi\ts_m^{-1/2}) \Eq O\bigl(Ld^{7/2} (m/\ts_m)^{1/2}\Lbar^{3/2}\bigr).
\]
Combining this with the remaining elements of the bound given in Theorem~\ref{AX-exch-thm},
and noting that $\ts_m \le m$, the theorem follows.
\end{proof}

Except for the logarithmic factors, the bound obtained in the theorem is of the same order in~$m$ as would 
be expected for weaker metrics, such as the convex sets metric (Bentkus 2003, Fang \& R\"ollin 2015), 
if $\ts_m \asymp m$.  The latter asymptotic
equivalence holds, for example, for identically distributed summands 
whose common distribution has non-trivial overlap with its unit translates in each direction.
It is possible, however, for $\ts_m$ to be significantly smaller than~$m$.  For instance, if
all the summands making up~$W$ are on~$2\Z \times \Z^{d-1}$, then $s_m = 0$, and the discrete normal is not a good
approximation to~$W$ in total variation, since it puts about half its probability mass on points
whose first coordinate is an odd integer, whereas~$\law(W)$ puts zero mass on this set.
  
The best approximation order
with respect to the convex sets metric, for sums of independent and identically distributed random
variables with finite third moment, is $O(d^{7/4}L)$.  Thus our rate is weaker in~$m$ by 
a factor of~$\log m$, and in dimension by a factor \adbr{of~$d^{9/4}$.} % if \adbb{$\gbar(\VVV/m)/\Lbar$} is bounded. 
If the distributions are not identical,  the best known $d$-dependence
for approximation in the convex sets metric
is rather worse, unless the random variables are also assumed to be bounded.
% --- for independent and identically distributed random
% variables, $n$ and~$\ts_m$ are typically asymptotically equivalent.  
Since the total variation metric is substantially stronger than the convex sets metric, 
our bounds are of encouragingly small order \adbr{in~$d$, too}.

\subsection{Exchangeable pairs}\label{exchangeable}
If the pair $(W,W')$ is also {\it exchangeable\/}, so that
$\law((W,W')) = \law((W',W))$, a neat argument of R\"ollin \& Ross~(2015) delivers 
bounds on the quantities $\e_1$ and~$\e_1(\xi)$ of~\Ref{Exch-dtv-defs}, which appear in
the bound given in Theorem~\ref{AX-exch-thm}.
These can be of considerable practical use in deriving 
explicit bounds from the general expressions given in Theorem~\ref{AX-exch-thm}.

\adbb{For $\xi := W'-W$},
let~$\JJ$ be the set of $d$-vectors such that $q^J := \pr[\xi=J] > 0$,
and suppose \adbr{that each of the coordinate \adbb{vectors~$\ej\in\re^d$} %, $1\le j\le d$,
can be obtained as a (finite) sum of elements of~$\JJ$.}
For $Q^J(W) := \pr[\xi = J \giv W]$,  set
\eq\label{Exch-cdl-prob-defs}
  \vva^J \Def (q^J)^{-1} \ex|Q^J(W) - q^J|,
\en  
to be thought of as small. Note that, by exchangeability,
\eq\label{Exch-prob-symmetry}
    q^J \Eq \ex\{I[W'-W = J]\} \Eq \ex\{I[W-W' = J]\} \Eq q^{-J}.
\en
We then write
\[
   \vvb_j \Def  \sum_{l=1}^{r(j)} (\vva^{J\uj_l} + \vva^{-J\uj_l}),
       \adbr{\quad\mbox{where}\quad \sum_{l=1}^{r(j)}J\uj_l \Eq \ej,}
\]
% where the $(J_l\uj,\,1\le l\le r(j),\,1\le j\le d)$ are as in Assumption~\adbg{G4},
and then set $\vvb^* := \max_{1\le j\le d}\vvb_j$ and $\vva^* := \sup_{J\in\JJ}\vva^J$.
With the help of these quantities, we can bound the differences $\dtv(\law(W),\law(W+\ej))$
between the distribution of~$W$ and its translates.  

\begin{lemma}\label{Exch-dtv-lemma}
For each $1\le j\le d$, we have
\[
    \dtv(\law(W+\ej),\law(W)) \Le  \vvb_j,
\]
and
\[
   \dtv(\law(W + \ej \giv \xi = J),\law(W \giv \xi = J)) 
          \Le \vvb_j + 2\vva^J.
\]
Hence, in particular, for each $J\in\JJ$,
\[
  \dtv(\law(W + \ej \giv \xi = J),\law(W \giv \xi = J)) 
          \Le \vvb^* + 2\vva^*,
\]
and $\dtv(\law(W+\ej),\law(W)) \le  \vvb^*$.
Furthermore, for~$R_2(W)$ as defined in~\Ref{Exch-props}, we have
\[
    \ex\|R_2(W)\|_1   \Le d \tr(\s^2) \vva^*.
\]\end{lemma}

\begin{proof}
 For any $J \in \JJ$ and any~$f$ with $\|f\|_\infty = 1$, we use exchangeability to give
\[
    \ex\{f(W')I[W'-W = J] - f(W)I[W-W' = J]\} \Eq 0.
\]
\adbr{As in the proof of Theorem~3.6 of R\"ollin \& Ross~(2015),
we divide by~$q^J$, using~\Ref{Exch-prob-symmetry}, and
evaluate the expectation by conditioning on~$W$, giving}
\eqs
   0 &=& (q^J)^{-1}\ex\{f(W+J)Q^J(W) - f(W) Q^{-J}(W)\} \\
     &=& \ex\{f(W+J) - f(W)\} + (q^J)^{-1}\ex\{f(W+J)(Q^J(W) - q^J)\} \\
     &&\qquad\mbox{}      - (q^{-J})^{-1}\ex\{f(W)(Q^{-J}(W) - q^{-J})\},
\ens
from which it follows that
$$
    \dtv(\law(W+J),\law(W)) \Le \vva^J + \vva^{-J}.  
$$
The first statement now follows by the triangle inequality.

For the second, % \adbr{writing $\xi := W'-W$,} 
we have
\eqs
   \leqn{ \ex\{f(W+\ej) - f(W) \giv \xi = J\} }\\
         &=& (q^J)^{-1}\ex\{(f(W+\ej) - f(W))I[\xi = J]\} \\
    &=& (q^J)^{-1}\ex\{(f(W+\ej) - f(W))Q^J(W)\} \\
    &=& \ex\{f(W+\ej) - f(W)\} \\
    &&\qquad\mbox{} +  (q^J)^{-1}\ex\{(f(W+\ej) - f(W))(Q^J(W) - q^J)\}.\phantom{XXX} 
\ens
Hence we have
\eqa
  \leqn{\dtv(\law(W \giv \xi = J),\law(W+\ej \giv \xi = J)) } \non\\ 
   &&\Le \dtv(\law(W),\law(W+\ej)) + 2\vva^J, \label{Exch-condl-dtv}
\ena
and the second part follows; note that exchangeability was not used in proving~\Ref{Exch-condl-dtv}.

Finally, from the definition of~$R_2(W)$ in~\Ref{Exch-props}, we have
\[
    \{R_2(w)\}_{il} \Eq  \s^2_{il}(w) - \s^2_{il} \Eq \sum_{J,J' \in \JJ} J_i J'_l (Q^J(w) - q^J),
\]
for any $1\le i,l \le d$, so that
\[
    \ex|\s^2_{il}(W) - \s^2_{il}| \Le \sum_{J,J' \in \JJ} q^J |J_i|\, |J'_l| \vva^J
                                  \Le \ex\{|\xi_i|\,|\xi_l|\}\vva^*.
\]
This in turn implies that
\[
  \ex\|R_2(W)\|_1 \Eq \sid \sld \ex|\s^2_{il}(W) - \s^2_{il}| \Le \ex\{\nl{\xi}^2\}\vva^*
              \Le d \tr(\s^2) \vva^*,
\]
as claimed.
\end{proof}

The following corollary is immediate.
\begin{corollary}\label{Dec-useful-bounds}
Under the above assumptions,
\[
      d^{1/2}\ex\|R_2(W)\|_1 \Le C\Lbar \tn^{-1/2} d^{5/2}\{\tn^{1/2} \vva^*\}
\]
and 
\[
      d^2\ex\{|\xi|^3\e_1(\xi)\} \Le C'\Lbar^{3/2} d^{7/2} L \{\tn^{1/2}(\vvb^* + 2\vva^*)\},
\]
for constants $C$ and~$C'$ that depend only on~$\Sp'(\s^2/\Lbar)$. $\hfill\qed$
\end{corollary}

\begin{remark}\label{Exch-conditioning}
{\rm
Note that, by the argument in Remark~\ref{Exch-sigma-fields}, we can bound the quantities~$\vva^J$
above by $(q^J)^{-1} \ex|\pr[\xi = J \giv \FF] - q^J|$, for any $\s$-field~\adbr{$\FF$} such that~$W$ is $\FF$-measurable.
Such quantities may be easier to bound in practice.
}
\end{remark}

\begin{remark}
{\rm
For an exchangeable pair $(W,W')$, we see that
\eqa
  \ex\{\xi\xi^T\} &=& \ex\{(W'-\m)(W'-W)^T - (W-\m)(W'-W)^T\} \non\\
       &=& -\ex\{-(W-\m)(W-W')^T + (W-\m)(W'-W)^T\} \non\\
   &=& -2\ex\{(W-\m)(W'-W)^T\} \Eq -2\ex\{(W-\m)\ex(\xi^T \giv W)\} \non\\
        &=& -2\ex\{\ex(\xi \giv W)(W-\m)^T\}, \label{Exch-A-symmetry}
\ena
the last equality following because $\ex\{\xi\xi^T\}$ is symmetric.  If the
remainders~$R_1(W)$ and~$R_2(W)$ in~\Ref{Exch-props} were exactly zero, this would
give
\[
   \half\s^2 \Eq -n^{-1}A\, {\rm Cov}(W) \Eq -n^{-1}{\rm Cov}(W) A^T,
\]
and hence also
\adbr{$A^{-1}\s^2 = \s^2 (A^T)^{-1}$.}  If this is the case, we can easily solve for~$\S$, since then
$\S := -\half A^{-1}\s^2 = -\half \s^2(A^T)^{-1}$ satisfies
$A\S + \S A^T + \s^2 = 0$ and is symmetric. 
}
\end{remark}

\subsubsection{Monochrome edges in regular graphs}\label{Dec-monochrome}
As an example of the application of Theorem~\ref{AX-exch-thm} in the exchangeable setting,
suppose that~$G_n$ is an $r$-regular graph on~$n$ vertices (so that one of $n$ and~$r$
is even); thus there are $nr/2$ edges in the graph.
Let the vertices be coloured independently, each with one of~$m$ colours, the
probability of choosing colour~$i$ being~$p_i > 0$, $1\le i\le m$.  Let~$N_i$ denote the
number of vertices having colour~$i$, and let~$M_i$ denote the number of edges joining
pairs of vertices that both have colour~$i$.  We  approximate the
joint distribution of
\[
    W \Def (M_1,\ldots,M_m,N_1,\ldots,N_{m-1}) \ =:\ (W_1,\ldots,W_m,W_{m+1},\ldots,W_{2m-1}),
\]
when~$n$ becomes large, while $r$, $m$ and $p_1,\ldots,p_m$ remain fixed; the detailed
structure of~$G_n$ does not appear in the approximation.
Of course, the value of $N_m = n - \sum_{i=1}^{m-1}N_i$ is implied by knowledge of~$W$.
This problem, in the context of multivariate normal approximation, was considered by
Rinott \& Rotar~(1996) and in Chen, Goldstein \& Shao~(2011, pp.333--334).

\begin{theorem}\label{monochrome-example}
 For $m \ge 3$, $r$ and $p_1,\ldots,p_m$ fixed, we can find $\n \in \re^{2m-1}$ and a 
$(2m-1) \times (2m-1)$ covariance matrix~$\S$ such that, as $\nti$,
\[
    \dtv(\law(W),\DN_{2m-1}(n\n,n\S)) \Eq O(n^{-1/2}\log n).
\]
\end{theorem}

\begin{proof}
We use the notation of Theorem~\ref{AX-exch-thm} throughout.
We begin by observing that 
\[
     \ex M_i \Eq nrp_i^2/2;\quad \ex N_i \Eq np_i,
\]
determining $\n := n^{-1}\ex W$. After rather more calculation, the covariances
are given, for $1\le i\neq l\le m$, by
\eq\label{mono-cov}
  \begin{array}{ll}
   \var(M_i) \Eq \half nr p_i^2(1-p_i)\{1 + (2r-1)p_i\}; &\cov(N_i,N_l) \Eq -np_ip_l;\\
   \cov(M_i,M_l) \Eq -\half nr(2r-1)p_i^2p_l^2;  &\cov(M_i,N_l) \Eq -nrp_i^2p_l;\\
    \cov(M_i,N_i) \Eq nrp_i^2(1-p_i); &\var(N_i) \Eq np_i(1-p_i),
  \end{array}
\en
in turn determining~$\S$.

 We now construct an exchangeable pair $(W,W')$ by
first realizing a colouring $(C(j),1\le j\le n)$, and using it to define
\eq\label{Sept-MN-defs}
   M_i \Def \sum_{\{j,j'\} \in G}I[C(j) = C(j')=i]\quad\mbox{and}\quad 
      N_i \Def \sjn I[C(j) = i],
\en
for each $1\le i\le m$, thus defining~$W$.  We then choose a vertex~$K$ uniformly
at random, independently of $(C(j),1\le j\le m)$, and then replace $C(K)$ by~$C'$,
where~$C'$ is independently sampled from $1,2,\ldots,m$ with $\pr[C'=i] = p_i$,
$1\le i\le m$. If this new colouring is denoted by $(C'(j),1\le j\le m)$, then
we define $M_i'$ and~$N_i'$ as in~\Ref{Sept-MN-defs}, but with the~$C'(j)$ in place of~$C(j)$,
and hence deduce~$W'$.  
Of course, $\law(W,W') = \law(W',W)$, and~$W'$ differs
from~$W$ only through the (possibly) new colour at the vertex~$K$, and through
its impact in changing which edges incident to~$K$ are monochrome:
\eqs
    M_i' - M_i &=&   \sum_{j\colon \{j,K\} \in G}(I[C(j) = C'(K) = i] - I[C(j) = C(K) = i]) \\   
     N_i' - N_i &=& \{I[C'(K) = i] - I[C(K)=i]\}.
\ens
Hence, for $1\le l\le m$, we have
\eqs
    \leqn{\ex\{\xi_l \giv C(1),\ldots,C(n)\}}\\
        &=& n^{-1}\skn \sum_{j\colon \{j,k\} \in G}\{p_l I[C(k) = l] - I[C(j)=C(k)=l]) \\
        &=& n^{-1}\{p_l r N_l - 2M_l\} \Eq \ex\{\xi_l \giv W\},
\ens
and, for $m+1 \le l \le 2m-1$,
\[
    \ex\{\xi_l \giv C(1),\ldots,C(n)\} \Eq n^{-1}\{n p_{l-m} - N_{l-m}\} \Eq \ex\{\xi_l \giv W\}.
\]
This gives an exact linear regression as in~\Ref{Exch-props}, with~$R_1(w) = 0$ for all~$w$, and with~$A$ 
having non-zero elements given by
\eqs
   A_{ll} &:=& -2,\quad A_{l,l+m} \Def rp_l,\qquad 1\le l\le m-1; \\
   A_{mm} &:=& -2,\qquad A_{m,m+t} \Def -rp_m, \quad 1 \le t \le m-1;\\
   A_{ll} &:=& -1,\qquad m+1 \le l \le 2m-1.
\ens
Since~$A$ is upper triangular, its eigenvalues are $-2$, with multiplicity~$m$, and~$-1$, with multiplicity~$m-1$,
so that it is indeed spectrally negative.

The set~$\JJ$, consisting of the possible values that can be taken by~$\xi$, is finite, and
does not depend on~$n$. If $C(K) = i \neq l = C'(K)$, then
the $m+i$ and~$m+l$ components of~$\xi$ each have modulus one (though, if $i$ or~$l$ are equal to~$m$, 
one of these components is not present in~$W$), and the $i$ and~$l$ components are in modulus at most~$r$;
all other components of~$\xi$ are zero.  Hence $|\xi|^2 \le 2(r^2+1)$ a.s., and $\ex|\xi|^3$ remains
bounded as~$n$ increases; $L$ is thus of strict order~$n^{-1/2}$.  The components of~$\s^2:= \ex\{\xi\xi^T\}$ can be 
explicitly calculated: for $1 \le l \neq l' \le m$, they are given by
\eqs
  \ex \xi_l^2 &=& 2p_l^2(1-p_l)\{r(r-1)p_l + r\};\quad \ex\{\xi_l\xi_{l'}\} \Eq -2r(r-1)p_l^2p_{l'}^2;\\
  \ex\{\xi_l\xi_{m+l}\} &=& 2rp_l^2(1-p_l);\quad \ex\{\xi_l\xi_{m+l'}\} \Eq -2rp_l^2p_{l'};\\
  \ex\{\xi_{m+l}^2\} &=&  2p_l;\quad \ex\{\xi_{m+l}\xi_{m+l'}\} \Eq -2p_l p_{l'},
\ens
where terms with subscript~$2m$ are to be ignored.
% The matrix~$\S$ can then in principle be deduced by solving the equation $A\S + \S A^T + \s^2 = 0$.

In order to apply Theorem~\ref{AX-exch-thm}, we now just need to find bounds for $\e_1$, $\ex\{|\xi|^3\e_1(\xi)\}$
and $\ex\|R_2(W)\|_1$. From Lemma~\ref{Exch-dtv-lemma} and Corollary~\ref{Dec-useful-bounds}, 
these are all
bounded by fixed multiples of $\vva^*$ and~$\vvb^*$.  For each~$J$ in the fixed finite set~$\JJ$,
the probability~$q^J$ in the denominator of~$\vva^J$ is fixed and positive, and hence 
bounded away from zero. To bound the numerator, we condition on a larger $\s$-field $\FF$,
with respect to which~$W$ is measurable, as in Remark~\ref{Exch-conditioning}.  Let~$T_{m,r}$ denote the
set of all $m$-tuples of nonnegative
integers $t_1,\ldots,t_m$ such that $\siim t_i = r$, and, for $\bt := (t_1,\ldots,t_m) \in T_{m,r}$, 
let $E_j(i_0;\bt)$ denote the
event that $C(j) = i_0$, and that~$t_i$ of the~$r$ neighbours of~$j$ have colour~$i$, $1\le i\le m$.
For each fixed~$j$, these are disjoint events whose union over~$1\le i_0 \le m$ and $\bt\in T_{m,r}$ 
is the sure event.
We let~$\FF$ be the $\s$-field generated by the events 
\[
   \{E_j(i_0;\bt);\ 1\le j\le n, 1\le i_0\le m, \bt\in T_{m,r} \}.
\]
Then, if $K = j$, the value~$J \in \JJ$ taken by~$\xi$ is determined by which of the events
$(E_j(i_0;\bt);\ 1\le i_0\le m, \bt\in T_{m,r})$ occurs. For each~$J$, there is a collection~$S(J)$ 
of possible choices, consisting of just one possible~$i_0 = i_0(J)$, the index for which~$J_{m+i_0} = -1$
(if there is none, then $i_0 = m$), but of all~$\bt$ that satisfy $t_{i_0} = -J_{i_0}$
and $t_{i_1} = J_{i_1}$, where~$i_1$ is the index for which $J_{m+i_1} = 1$ (or~$m$, if there is
none such). Thus
\[
    \pr[\xi = J \giv \FF] \Eq n^{-1}\sjn\, \sum_{\bt \in T_{m,r}\colon\! (i_0(J),\bt) \in S(J)} I[E_j(i_0(J);\bt)].
\]
Now, if $j'\neq j$ is such that the set of neighbours~$\NN(j)$ (including~$j$) in~$G$ is disjoint from 
the set~$\NN(j')$, the events $I[E_j(i_0(J);\bt)]$ and~$I[E_{j'}(i_0(J);\bt')]$ are independent.  Since, for
each~$j$, there are no more than $r + r^2$ choices of~$j' \neq j$ for which this is not the case, it follows
that
\[
     \var\{\pr[\xi = J \giv \FF]\} \Eq O(n^{-1}).
\]
Hence $\var\{Q^J(W)\} = O(n^{-1})$ also, and so $\ex|Q^J(W) - q^J| = O(n^{-1/2})$ for all~$J \in \JJ$,
implying that $\vva^* =  O(n^{-1/2})$.

The argument for~$\vvb^*$ is not yet finished, since, for each $1\le l\le 2m-1$, it is necessary to find a 
chain $J\ui,J\ut,\ldots,J^{(R)}$ such that each $J\uii \in \JJ$ and $\sum_{i=1}^R J\uii = e^{(l)}$.
For $m+1 \le l \le 2m-1$, this is easy: $\xi = e^{(l)}$ if, when~$W$ is constructed, a vertex has colour~$m$ 
and no neighbours of colours $m$ or~$l$, and its colour is replaced by~$l$ when resampling to obtain~$W'$.
Note that, to do this, we need at least three colours: $m \ge 3$. To get $e^{(l)}$ for $1 \le l\le m-1$,
a chain of length~$2$ is needed: a vertex of colour~$m$ with no neighbours of colour~$m$ and with exactly one of
colour~$l$ is recoloured with colour~$l$, giving $J = e^{(l)} + e^{(l+m)}$.  Then $J = -e^{(l+m)}$ can be
attained by reversing the order of the choices in the example for $m+1 \le l \le 2m-1$.  To get $e^{(m)}$,
a vertex of colour~$l \neq m$ with no neighbours of colour~$l$ and exactly one of colour~$m$ is recoloured~$m$,
yielding $e^{(m)} - e^{(m+l)}$, and then adding $e^{(m+l)}$ as before completes the chain.  Thus, 
for $m \ge 3$, we have $\vvb^* = O(n^{-1/2})$ also, and applying Theorem~\ref{AX-exch-thm}, the result follows.
\end{proof}
  
There remains the case of $m=2$.  Here, discrete normal approximation in total variation
is not good, since it can be seen that $M_1 - M_2 = r(N_1 - n/2)$, so that~$W$ is degenerate;
what is more, reducing to $(W_1,W_2)$ gives an integer vector living on a proper sub-lattice of~$\Z^2$.
However, the pair $(M_1,N_1)$ can be approximated using the method above, and the remaining
components of $M$ and~$N$ follow from $N_2 = n - N_1$ and $M_2 = M_1 - r(N_1-n/2)$.

\section{Technicalities}\label{appendix}
\subsection{Proof of Lemma~\ref{AX-first-lemma}}\label{L5.1}
Let~$\f_n$ denote the density of the multivariate normal distribution
$\NN_d(nc,n\S)$, and, for $X \in \Z^d$, let $[X]$ denote the box
\[
   [X] \Def \bigl\{x \in \re^d\colon X_i - \half < x_i \le X_i + \half,\,1\le i\le d \bigr\}.
\]
\adbr{Let $N_d$, $d\ge1$, denote a standard $d$-dimensional normal random vector.}
For~(a), the bound on $\ex \nS{W-nc}^l$ is obtained by first writing 
$$
   \nS{X-nc}^{l} \Le (\nS{X-t} + \nS{t-nc})^{l} \Le 2^{l-1}(\nS{X-t}^{l} + \nS{t-nc}^{l}).
$$
Taking this inside the integral, we have
\eqs
  \ex \nS{W-nc}^l &=& \sum_{X \in \Z^d} \nS{X-nc}^{l} \int_{[X]} \f_n(t) \,dt  \\
   &\le&  \sum_{X \in \Z^d} \int_{[X]} \f_n(t) 
            2^{l-1}((\half\sqrt{d/\lmin(\S)})^{l} + \nS{t-nc}^{l}) \,dt \\
   &\le&  \ex\{2^{l-1}((\half\sqrt{d/\lmin(\S)})^{l} + n^{l/2}|N_d|^{l}) \} \\[1ex]
   &\le&  2^{l} \ex|N_d|^{l} n^{l/2},
\ens
for 
$$
  n\ \ge\ \frac{d}{4(\ex|N_d|)^2\lmin(\S)} \Eq \frac d{8\lmin(\S)}\{\G(d/2)/\G((d+1)/2)\}^2.
$$ 
Part~(a) follows, taking \adbr{$C(l) := 2^{l}\sqrt{k(l)}$, where
$$
     k(l) \Def \ex N_1^{2l} \Eq  \frac{(2l)!}{2^l l!},
$$
since $2^{l} \ex|N_d|^{l} \le 2^l\sqrt{\ex N_d^{2l}} \Le 2^l d^{l/2} \sqrt{\ex N_1^{2l}}$,}
and by noting that, in $d\ge1$,
$\tfrac d8\{\G(d/2)/\G((d+1)/2)\}^2 \le 1$.

For~(c), the bound on $\ex\{[\S^{-1}(W-nc)]_j^{2l}\}$, we first note \adbr{that
\[
    \ex\{(a^T N_d)^{2l}\} \Eq (a^Ta)^l \ex\{N_1^{2l}\} 
                          \Eq  k(l)(a^Ta)^l , 
\]
for any $a \in \re^d$.}  So, since
\[
    [\S^{-1}(X-nc)]_j^{2l} \Le 2^{2l-1}\Blb  \Bl\frac{ d}{4\{\lmin(\S)\}^2}\Br^l + [\S^{-1}(t-nc)]_j^{2l} \Brb
\]
for $t \in [X]$, it follows that
\eqs
    \leqn{\ex\{[\S^{-1}(W-nc)]_j^{2l}\} \Eq \sum_{X \in \Z^d} [\S^{-1}(X-nc)]_j^{2l} \int_{[X]} \f_n(t)\,dt }\\
     &&\Le 2^{2l-1}\Blb  \Bl\frac{ d}{4\{\lmin(\S)\}^2}\Br^l 
                     + n^l\ex\bigl\{\{(\ej)^T\S^{-1/2}N_d\}^{2l}\bigr\} \Brb \\
     &&\Le 2^{2l-1}\Blb  \Bl\frac{ d}{4\{\lmin(\S)\}^2}\Br^l + n^l k(l) (\S^{-1})_{jj}^l \}, \Brb
\ens  
and the stated bound follows, \adbr{with $C'(l) = 2^{2l-1}k(l)$}.  Part~(b) is similar, but simpler. \hfill$\qed$

\subsection{Proof of Lemma~\ref{AX-main-lemma}}\label{L5.2}
We note first that, from Lemma~\ref{AX-first-lemma}(a),
\eq\label{AX-DN-moments}
   \ex\nS{W-nc}^i \Le C(i) (nd)^{i/2},
\en
if $n \ge 1/\lmin(\S)$.
For~(a), bounding the difference between $\ex\{\D f(W)^T b\, I\nud(W)\}$
and $n^{-1}\ex\{(f(W)\,(W-nc)^T\S^{-1}b\, I\nud(W)\}$, we begin by observing that
\eqa
   \lefteqn{\ex\{\D_j f(W) I\nud(W)\} \Eq \sum_{X\in\Z^d} f(X)\{\pr[W = X-\ej] - \pr[W=X]\} I\nud(X)} \non\\
     &&\mbox{} + \sum_{X\in\Z^d} f(X) \pr[W = X-\ej]\{I\nud(X-\ej) - I\nud(X)\}.\phantom{XXXXXX}
     \label{AX-part-one-1}
\ena
Because, from the definition of~$I\nuh(X)$, $|I\nud(X-\ej) - I\nud(X)| = 1$ requires $\nS{X-\ej-nc} > n\d/3$
and $\nS{X-nc} \le n\d/3$, or vice versa, the last term in~\Ref{AX-part-one-1} is in modulus at most 
$$
   \pr[\nS{W -nc} > n\d/3-1/\sqrt{\lmin(\S)}] \max_{\nS{X-nc} \le n\d/3+1/\sqrt{\lmin(\S)}}|f(X)|.
$$
Thus it follows from~\Ref{AX-DN-moments} and a fourth moment Markov inequality that, 
if $n \ge  \max\{1/\lmin(\S),6/(\d\sqrt{\lmin(\S)})\}  = \max\{1/\lmin(\S),\smh_\S(\d)\}$ , then
\eqa
  \lefteqn{\Bigl| \sum_{X\in\Z^d} f(X) \pr[W = X-\ej]\{I\nud(X-\ej) - I\nud(X)\}\Bigr|}\non\\
     &&\Le \|f\|^\S_{n\d/\adbg{2},\infty}\, \pr[\nS{W-nc} > n\d/6] \non\\
     &&\Le (6/\d)^4 d^2 C(4)n^{-2} \|f\|^\S_{n\d/\adbg{2},\infty} 
         \Le d^2C_1(\d) n^{-2} \|f\|^\S_{n\d/\adbg{2},\infty}, \label{AX-err-bnd-1}
\ena
where $C_1(\d) = (6/\d)^4 C(4) \in \KKS(\d)$.  
 
For the remainder of~\Ref{AX-part-one-1}, we write
\eqs
  \pr[W = X-\ej] - \pr[W=X] &=& \int_{[X]} \f_n(t) D_j(t)\,dt,
\ens
where
\[
    D_j(t) \Def \exp\Blb -\frac1{2n}\{-2[\S^{-1}(t-nc)]_j + (\S^{-1})_{jj}\}\Brb - 1.
\]
Since $|e^x - 1 - x| \le \half x^2 e^{|x|}$, it follows that, for $\nS{X-nc} \le n\d/3$,
\eqs
  \leqn{\Bigl| D_j(t) - \frac1n [\S^{-1}(t-nc)]_j \Bigr|} \\
    &&\le \frac1{2n}|(\S^{-1})_{jj}| + \frac1{n^2}\{([\S^{-1}(t-nc)]_j)^2 + \quarter (\S^{-1})_{jj}^2\}e^{\ccc_j(\d)},
    \phantom{XX}
\ens
where
\eqa
  \ccc_j(\d) &:=& \frac1n\max_{\nS{X-nc} \le n\d/3}\Bigl\{
         |[\S^{-1}(X-nc)]_j| +  |(\S^{-1})_{jj}| + \half d^{1/2} \|\S^{-1}\| \Bigr\} \non\\
   &\le& \frac13 \|\S^{-1/2}\| \d + \frac3{2\lmin(\S)} \ =:\ \ccc^*(\d), \label{Aug-cj-def}
\ena 
if $n \ge d^{1/2}/\lmin(\S)$, true in turn if 
$n \ge n_1 := (\lmin(\S))^{-8/7}$, because $n \ge d^4$. Note also that $n_1 \ge 1/\lmin(\S)$.
Hence, fixing~$\d$, for such~$X$ and for $t \in [X]$,
\eqa
  \Blm D_j(t) - \frac1n [\S^{-1}(X-nc)]_j \Brm
    &\le& C_2(\d) n^{-1}(d^{1/2} + n^{-1}[\S^{-1}(X-nc)]_j^2),\phantom{XX} \label{AX-Dj-approx}
\ena
for~$C_2(\d) := 2e^{\ccc^*(\d)}/\lmin(\S) \in \KKS(\d)$, again if $n \ge n_1$. 
This in turn implies that, for $\nS{X-nc} \le n\d/3$,
\eqa
  \leqn{\bigl| \{\pr[W = X-\ej] - \pr[W=X]\}  - n^{-1}\pr[W=X][\S^{-1}(X-nc)]_j \bigr|} \non\\
 &&\Le    C_2(\d) n^{-1}(d^{1/2} + n^{-1}[\S^{-1}(X-nc)]_j^2)\pr[W=X], \label{AX-prob-diff}\phantom{XXXXX}
\ena
and hence that
% and, using~\Ref{AX-f-norm-bnd}, this gives
\eqa
   \lefteqn{\Bigl| \sum_{X\in\Z^d} f(X)\{\pr[W = X-\ej] - \pr[W=X]\} I\nud(X)}\non\\
    &&\qquad\qquad\qquad\qquad\mbox{} -
      n^{-1}\ex\{f(W)[\S^{-1}(W-nc)]_jI\nud(W)\}\Bigr| \phantom{XXX}\non\\
          &&\Le  C_2(\d)n^{-1}\ex\{d^{1/2} + n^{-1}[\S^{-1}(W-nc)]_j^2\}\|f\|^\S_{n\d/2,\infty}.\label{AX-part-one-2}
\ena
% for $C_2' \in \KKS$.
Now, writing $b = \sum_{j=1}^d b_j\ej$ and using linearity and Lemma~\ref{AX-first-lemma}(c), 
requiring $n \ge d/\{4(\lmin(\S))^2\}$, the inequality~(a) follows, if 
$n \ge \max\{n_{\ref{AX-main-lemma}},\smh_\S(\d)\}$, where
\eq\label{Dec-n5.2-def}
     n_{\ref{AX-main-lemma}} \Def \max\Blb d^4, n_1,\{4(\lmin(\S))^2\}^{-4/3}\Brb,
\en
with 
\eq\label{Dec-C5.2(1)-def}
      C\ui_{\ref{AX-main-lemma}}(\d) \Def C_1(\d) + C_2(\d)\{1 + C'(1)(1+1/\lmin(\S))\}.
\en

For~(b), bounding the difference between $\ex\{\D f(W)^T B(W-nc)\, I\nud(W)\}$
    and $\ex\{f(W)\,[n^{-1}(W-nc)^T \S^{-1}B(W-nc) - \tr B]\, I\nud(W)\}$,  we argue in similar style.  
For $i \neq j$, writing $E^{(ji)} := \ej (\eii)^T$, we have
\eqa
  \leqn{ \ex\{\D f(W)^T E^{(ji)}(W-nc) I\nud(W)\} \Eq \ex\{\D_j f(W)(W_i-nc_i) I\nud(W)\} }\non\\
         &=& \sum_{X\in\Z^d} f(X)(X_i-nc_i)\{\pr[W = X-\ej] - \pr[W=X]\} I\nud(X) \label{AX-part-two-1}\\
     &&\mbox{}\quad + \sum_{X\in\Z^d} f(X)(X_i-nc_i) \pr[W = X-\ej]\{I\nud(X-\ej) - I\nud(X)\}.
     \non
\ena
% Using~\Ref{AX-f-norm-bnd}, 
For $n \ge \max\{n_{\ref{AX-main-lemma}},\smh_\S(\d)\}$,
we bound the second element in~\Ref{AX-part-two-1} much as for~\Ref{AX-err-bnd-1},
using a Markov inequality, Cauchy--Schwarz and Lemma~\ref{AX-first-lemma}(a,b), giving
\eqa
  \lefteqn{\Bigl| \sum_{X\in\Z^d} f(X)(X_i-nc_i) \pr[W = X-\ej]\{I\nud(X-\ej) - I\nud(X)\}\Bigr|}
                 \non\\
         &&\Le \ex\{|W_i - nc_i| I[\nS{W-nc} > n\d/6]\}  \|f\|^\S_{n\d/2,\infty}\non\\
         &&\Le (6/n\d)^3 \ex\{|W_i - nc_i| \nS{W-nc}^3\}  \|f\|^\S_{n\d/2,\infty}\non\\
         &&\Le (6/n\d)^3 \sqrt{\ex|W_i -nc_i|^2}\,\sqrt{\ex\nS{W-nc}^6}\, \|f\|^\S_{n\d/2,\infty}\non\\
% (6/\d')^4 C_d(4,0)n^{-1} (\d'\sqrt{\lmax(\S)}/3)\sqrt d\{1 + \sqrt{\r(\S)} + \third n\d'\sqrt{\lmax(\S)}\}
%                  \|\D f\|^\S_{n\d'/2,\infty} \non\\
     &&\Le (6/\d)^3 n^{-1}d^{3/2}\sqrt{2(1 + \S_{ii})C(6)}\, \|f\|^\S_{n\d/2,\infty} \non\\
      &&\Le d^{3/2}C_3(\d) n^{-1}\|f\|^\S_{n\d/2,\infty}, 
                \label{AX-err-bnd-2}
\ena
where~$C_3(\d) =  (6/\d)^3 \sqrt{2(1 + \lmax(\S))C(6)} \in\KKS(\d)$. 
 The first element in~\Ref{AX-part-two-1} is treated using~\Ref{AX-prob-diff}, Cauchy--Schwarz
and Lemma~\ref{AX-first-lemma}(b,c), giving
\eqa
   \lefteqn{\Bigl| \sum_{X\in\Z^d} f(X)(X_i-nc_i)\{\pr[W = X-\ej] - \pr[W=X]\} I\nud(X)}
           \non\\
    &&\qquad\qquad\qquad\mbox{} -
      n^{-1}\ex\{f(W)(W_i-nc_i)[\S^{-1}(W-nc)]_jI\nud(W)\}\Bigr| \phantom{XXX}\non\\
    &&\Le  C_2(\d) n^{-1}\ex\{|W_i-nc_i|(d^{1/2} + n^{-1}[\S^{-1}(W-nc)]_j^2)\}\|f\|^\S_{n\d/2,\infty} \non\\
    &&\Le C_2(\d) n^{-1/2} \sqrt{2(1 + \S_{ii})}\Bigl(d^{1/2}  + \sqrt{C'(2)(1+(\S^{-1})_{ii}^2)}\Bigr)
                    \|f\|^\S_{n\d/2,\infty} \non\\
    &&\Le  d^{1/2}C_4(\d) n^{-1/2}  \|f\|^\S_{n\d/2,\infty}, \label{AX-part-two-2}
\ena
with 
$$
    C_4(\d) \Def C_2(\d) \sqrt{2(1 + \lmax(\S)}\bigl(1  + \sqrt{C'(2)(1+\lmin(\S)^{-2})}\bigr)\ \in\ \KKS(\d).
$$ 
Note that 
$$
    (W_i-nc_i)[\S^{-1}(W-nc)]_j = (W-nc)^T\S^{-1}E^{(ji)}(W-nc).
$$

For $i=j$, there is an extra term:
\eqa
   \leqn{\ex\{\D f(W)^T E^{(ii)}(W-nc) I\nud(W)\}}  \non\\
         &&\Eq \sum_{X\in\Z^d} f(X)(X_i-nc_i)\{\pr[W = X-\eii] - \pr[W=X]\} I\nud(X) \non\\
     &&\mbox{}\quad + \sum_{X\in\Z^d} f(X)(X_i-nc_i) \pr[W = X-\eii]\{I\nud(X-\eii) - I\nud(X)\}.
                   \non\\
     &&\mbox{}\qquad - \sum_{X\in\Z^d} f(X) \pr[W = X - \eii] I\nud(X-\eii) . \label{AX-part-two-3}
\ena
Now
\eqs
    \leqn{ \sum_{X\in\Z^d} f(X) \pr[W = X - \eii] I\nud(X-\eii) }\\
         &&\qquad\Eq \ex\{\D_i f(W) I\nud(W)\} + \ex\{f(W) I\nud(W)\},
\ens
and $|\ex\{\D_i f(W) I\nud(W)\}| \le \|\D f\|^\S_{n\d/2,\infty}$, giving
\eqa
   \lefteqn{\Bigl| \ex\{\D f(W)^T E^{(ii)}(W-nc) I\nud(W)\} }\label{AX-part-two-4}\\
    &&\mbox{} -
      n^{-1}\ex\{f(W)(W-nc)^T\S^{-1}E^{(ii)}(W-nc) I\nud(W)\} - \ex\{f(W) I\nud(W)\}\Bigr| \non\\
          &&\Le  d^{3/2}C_3(\d) n^{-1}\|f\|^\S_{n\d/2,\infty} +  d^{1/2}C_4(\d) n^{-1/2}  \|f\|^\S_{n\d/2,\infty} +
            \|\D f\|^\S_{n\d/2,\infty}.\non
\ena
The second estimate now follows
for general~$B = \sid\sjd B_{ij}E^{(ij)}$, by linearity, with 
\eq\label{Dec-C5.2(2)-def}
     C\ut_{\ref{AX-main-lemma}}(\d) \Def C_4(\d)  + C_3(\d),
\en
provided that $n \ge d^2$.

The proof of the final part of Lemma~\ref{AX-main-lemma}, 
bounding the difference between $\ex\{\D f(W)^T B(W-nc)\, I\nud(W)\}$
and $\ex\{f(W)\,[n^{-1}(W-nc)^T \S^{-1}B(W-nc) - \tr B]\, I\nud(W)\}$,
proceeds in very much the same way, but starting with $\ej b^T$ in place of $E^{(ji)}$
in \Ref{AX-part-two-1} and~\Ref{AX-part-two-3}, for any $b \in \re^d$,
\adbr{and then writing $B = \sjd \ej b(j)^T$ with $b(j) := B^T \ej$.}  The quantities
$(X_i-nc_i)$ and $(W_i - nc_i)$ are replaced in the computations by $b^T(X-nc)$ and $b^T(W-nc)
= b^T \S^{1/2}\,\S^{-1/2}(W-nc)$ respectively.
The error terms corresponding
to \Ref{AX-err-bnd-2} and~\Ref{AX-part-two-2} then yield the bounds $d^2C_3'(\d)|b| n^{-1}\|f\|^\S_{n\d/2,\infty}$
and $dC_4'(\d)|b| n^{-1/2}  \|f\|^\S_{n\d/2,\infty}$, 
with 
\eqa
    C_3'(\d) &:=& (6/\d)^3 C(4) \sqrt{\lmax(\S)};  \label{ADB-new-C-dash-defs}\\
    C_4'(\d) &:=& C_2(\d)\sqrt{C(2)\lmax(\S)}\bigl\{1 + \sqrt{C'(2)(1 + \lmin(\S)^{-2})} \bigr\},
\ena
\adbb{giving $C_{\ref{AX-main-lemma}}\uh(d) = C_4'(\d)  + C_3'(\d)$.}
\adbr{
The analogue of~\Ref{AX-part-two-3} yields an error bounded by $|b_j|\|\D f\|^\S_{n\d/2,\infty}$, and
Part~(c) now follows.}
% by writing $B = \sjd \ej (\ej)^T B$, and applying the above bounds
% with~$b = (\ej)^T B$ for each $1\le j\le d$.  
\hfill$\qed$

\subsection{Proof of Lemma~\ref{Exch-moments}}\label{Exch-moments-appx}
To bound the moments of $Z := (nd\vnew)^{-1/2}\S^{-1/2}(W-\m)$, 
we use the equation $\ex h(W') - \ex h(W) = 0$ for suitably chosen real functions~$h$.
First, we take $h(w) = (w-\m)^T\S^{-1}(w-\m)$, giving
\[
   \ex\{2\xi^T\S^{-1}(W-\m) + \xi^T\S^{-1}\xi \} \Eq 0.
\]
\adbg{Noting that $\xi^T\S^{-1}\xi = \tr(\S^{-1/2}\xi\xi^T\S^{-1/2})$, and using~\Ref{Exch-props},} we have
\eqs
  \leqn{- \ex\{2n^{-1}(W-\m)^T\S^{-1}A(W-\m) + 2n^{-1/2}\adbn{\moda}^{1/2}R_1(W)^T\S^{-1}(W-\m)\} }\\
       &&\Eq \ex \tr(\s^2_\S(W)) \Eq \tr(\s^2_\S) , \phantom{XXXXXXXXXXXXXX}
\ens
where $\s^2_\S(W) := \S^{-1/2}\s^2(W)\S^{-1/2}$.
Writing $s_n^2 := \ex|Z|^2$, it follows from~\Ref{Exch-R-bnds-11} and because
$A\S + \S A^T + \s^2 = 0$ that 
\[
    2\lla_1 d\vnew \,s_n^2 \Le (d\vnew\lla_1)^{1/2} (\tr(\s^2_\S))^{1/2} (1 + s_n^2) + \tr(\s^2_\S) .
%                    \Eq \tr(\s^2)\{1 + 2\ps_1(1 + s_n^2)\}.
\]
From the definition \adbb{of~$\vnew$ in~\Ref{ADB-exch-defs-2}}, it thus follows directly that
% Hence, with $x := \lla_1^{-1}\|A\| \sqrt{\ex\{|R_1(W)|^2\}}$, it follows that
$s_n^2 \le  2$, establishing the first part.

For the third moment, we start with $h(z) =  (1+z^Tz)^{3/2}$.
The function~$h$ has derivatives
\[
    Dh(z) \Eq 3(1+z^Tz)^{1/2}z
\]
and
\[
            D^2h(z) \Eq \frac{3zz^T}{(1+z^Tz)^{1/2}} + 3(1+z^Tz)^{1/2}I.
\]
\adb{Furthermore, 
%its third derivatives are uniformly bounded.  Hence we have
\eqa
    \leqn{  \Bigl| \{h(z+\z) - h(z)\} - 3(1+z^Tz)^{1/2}\z^T z 
%   \\ &&\mbox{}\quad
    - \frac{3(\z^T z)^2}{2(1+z^Tz)^{1/2}} 
           - \frac32(1+z^Tz)^{1/2}|\z|^2 \Bigr|} \non\\[1ex]
     &&\ =:\ d_3(h,z,\z) \Le k_{3,h}|\z|^3,\phantom{XXXXXXXXXXXXXXXXXXXXX}\label{new-three-expn}
\ena
for a constant~$k_{3,h} \le 22$ that does not depend on~$d$. 
This can be seen  by considering separately the cases where $|\z| \ge (|z|\vee 1)$, $|\z| \le |z|$
and $1 \ge |\z| \ge |z|$.
}

\adb{
For $|\z| \ge (|z|\vee 1)$, simply take the terms one by one, 
% and use the mean value theorem on the first of them, 
giving
\[
    d_3(h,z,\z) \Le |\z|^3(\{5^{3/2} + 2^{3/2}\} + 3\cdot2^{1/2} + \frac32 + \frac32\cdot2^{1/2})
                  \Le 22|\z|^3 .
\]
For $1 \ge |\z| \ge |z|$,
use the bounds
\eqs
  |(1+x)^{1/2} - 1| &\le& \half x^{1/2};\quad |(1+x)^{3/2} - 1 - \tfrac32 x| \Le \tfrac38 x^{3/2}
\ens
in $0 \le x \le 1$ to give
\eqs
  |(1+z^Tz)^{1/2} - 1| &\le \half|\z|;\quad |h(z+\z) - h(z) - \tfrac32(2\z^Tz + \z^T\z)| \Le 27|\z|^3/8.
\ens 
Then the first, second and fourth terms in $d_3(h,z,\z)$ together give at most
\[
     |\z|^3(\tfrac{27}8 + \tfrac32 + \tfrac34) \Le \tfrac{45}8 |\z|^3,
\]
and the third adds at most $\frac32|\z|^3$ to this.
For $|\z| \le |z|$, Taylor's expansion gives
\eqs
  \Bigl|(1 + x + y)^{3/2} - (1+x)^{3/2} - \frac32 y(1+x)^{1/2} - \frac{3y^2}{8\sqrt{1+x}} \Bigr|
       &\le& \frac{|y|^3}{16(1+x)^{3/2}}.
\ens
We take $x = z^Tz$ and $y = 2\z^Tz + \z^T\z$, for which $|y| \le 3|\z||z|$.  
The first, second and fourth terms in $d_3(h,z,\z)$ together thus give 
\[
   \frac{3(2\z^Tz + \z^T\z)^2}{8(1+z^Tz)^{1/2}},
\]
up to an error of at most
\[
   \frac{|2\z^Tz + \z^T\z|^3}{16(1+z^Tz)^{3/2}} \Le \frac{27|\z|^3\,|z|^3}{16|z|^3} \Le \frac{27}{16}\,|\z|^3.
\]
Then
\[
   \Bigl| \frac{3(2\z^Tz + \z^T\z)^2}{8(1+z^Tz)^{1/2}} - \frac{3(\z^T z)^2}{2(1+z^Tz)^{1/2}} \Bigr|
      \Le \frac{12|\z|^3\,|z| + 3|\z|^4}{8|z|} \Le \frac{15}8\, |\z|^3,
\]
giving an overall bound of $\frac{57}{16}|\z|^3$.} 

\adbg{
We now substitute $z = Z = z(W)$  and $\z = (nd\vnew)^{-1/2}\S^{-1/2}\xi$  into~\Ref{new-three-expn}, 
and take expectations. Since 
$$
     \ex h(Z(W + \xi))  \Eq \ex h(Z(W)),
$$
this}
\adb{immediately gives 
\eqa
   \leqn{ \ex\{ -3(1+Z^TZ)^{1/2} (nd\vnew)^{-1/2}\xi^T\S^{-1/2}Z \}  }\non\\ 
   &&\Le n^{-1}\ex\Bigl\{ \frac{3Z^T\S^{-1/2}\xi\xi^T\S^{-1/2}Z}{2d\vnew(1+Z^TZ)^{1/2}} 
          + \frac32 (1+Z^TZ)^{1/2} \frac{|\S^{-1/2}\xi|^2}{d\vnew} \Bigr\} + \frac{k_{3,h}\chi_\S}{(nd\vnew)^{3/2}}\non\\
   &&\Le n^{-1}\frac{3}{2d\vnew}\ex\{(2|Z|+1)|\S^{-1/2}\xi|^2\}  + \frac{k_{3,h}\chi_\S}{(nd\vnew)^{3/2}} . \label{Sept-RE1}
\ena
Now  
\eqa
   \leqn{ \ex\{ -3(1+Z^TZ)^{1/2} (nd\vnew)^{-1/2}\xi^T\S^{-1/2}Z \} }\non\\ 
     &=& \ex\{ -3(1+Z^TZ)^{1/2} (nd\vnew)^{-1/2}(n^{-1}(W-\m)^TA^T + n^{-1/2}\adbn{\moda}^{1/2}R_1(W)^T) \S^{-1/2} Z \}\non\\
     &=&   n^{-1}\ex\{ 3(1+Z^TZ)^{1/2} (\half Z^T \s^2_\S Z - (d\vnew)^{-1/2}\adbn{\moda}^{1/2}R_1(W)^T \S^{-1/2} Z)\},
        \label{Sept-RE2}
\ena
and, using \Ref{Exch-R-bnds-12},
\eq\label{Sept-RE3}
   (d\vnew)^{-1/2}\adbn{\moda}^{1/2}\ex\{(1+Z^TZ)^{1/2}|R_1(W)^T \S^{-1/2} Z)|\} 
           \Le \quarter \lla_1 (1 + \ex|Z|^3).
\en
Then, by the \adbr{arithmetic and geometric means} inequality, for any $a > 0$,
\[
     |Z|\,|\S^{-1/2}\xi|^2 \Le \third\{ (a|Z|)^3 + 2(a^{-1/2}|\S^{-1/2}\xi|)^3\},
\]
so that, taking $a = (d\vnew\lla_1)^{1/3}$,
\eq\label{Sept-RE4}
    (d\vnew)^{-1}\ex\{|Z|\,|\S^{-1/2}\xi|^2 \} \Le \frac{\lla_1}3 \{\ex|Z|^3 + 2(n/\moda)^{1/2} L_\S\}.
\en
Combining \Ref{Sept-RE1}--\Ref{Sept-RE4}, recalling that $\tr(\s^2_\S) = d\vnew\lla_1$, and multiplying by~$n$, 
it follows that
\eqs
    3\lla_1 \ex|Z|^3 &\le& \tfrac34 \lla_1(1 + \ex|Z|^3) + \lla_1 \{\ex|Z|^3 + 2(n/\moda)^{1/2}L_\S\} \\
    &&\mbox{}\qquad\qquad\qquad   + \tfrac32 \lla_1  + k_{3,h}\lla_1 L_\S \sqrt{\frac{\lla_1}{\moda}},
\ens
giving $\ex|Z^3| \le 2(1 + 10(n/\moda)^{1/2}L_\S)$ if~$n /\lla_1 \ge 1$, because $k_{3,h} \le 22$.  
The final inequality is then immediate. \hfill$\qed$}


\begin{thebibliography}{00}
\ignore{
\bibitem{Bar88}
{\sc A. D. Barbour} (1988)
Stein's method and Poisson process convergence.
{\it J.\ Appl.\ Probab.\/} {\bf 25(A)}, 175--184.
%
\bibitem{BHJ92} 
{\sc A. D. Barbour, L. Holst \& S. Janson} (1992) 
{\it Poisson Approximation}.  Oxford Univ.\ Press.
%
\bibitem{BPII}
{\sc A. D. Barbour \& P. K. Pollett} (2012)
Total variation approximation for quasi-equilibrium distributions, II.
{\it Stoch.\ Procs.\ Applics\/} {\bf 122}, 3740--3756.
}

\bibitem{BLXI}
{\sc A.D. Barbour, M. J. Luczak \& A. Xia} (2017)
Multivariate approximation in total variation, I: equilibrium distributions of 
Markov jump processes.
{\tt arXiv:1512.07400}

\ignore{
\bibitem{BXia}
{\sc A. D. Barbour \& A. Xia} (1999)
Poisson perturbations.
{\it ESAIM: P\&S\/} {\bf 3}, 131--150.
} 

\bibitem{Bentkus03}
{\sc V.\ Bentkus} (2003)
On the dependence of the Berry--Esseen bound on dimension.
{\it J.~Statist.\ Plann.\ Inference\/} {\bf 113}, 385--402.

\ignore{
\bibitem{LHYC75}
{\sc L.\ H.\ Y.\ Chen} (1975)
Poisson approximation for dependent trials.
{\it Ann.\ Probab.\/} {\bf 3}, 534--545.
}

\bibitem{CGS11}
{\sc L.\ H.\ Y.\ Chen, L.\ Goldstein \& Q.-M.\ Shao} (2011)
{\it Normal approximation by Stein's method.}
Springer--Verlag, Berlin.

\bibitem{Fang14}
{\sc X.\ Fang} (2014)
Discretized normal approximation by Stein's method.
{\it Bernoulli\/}~{\bf 20}, 1404--1431.


\bibitem{FangRo}
{\sc X.\ Fang \& A.\ R\"ollin} (2015)
Rates of convergence for multivariate normal approximation with application to
dense graphs and doubly indexed permutation statistics.
{\it Bernoulli\/}~{\bf 21}, 2157--2189.

\bibitem{Goetze}
{\sc F.\ G\"otze} (1991)
On the rate of convergence in the multivariate CLT. 
{\it Ann.\ Probab.\/} {\bf 19}, 724--739.

\ignore{
\bibitem{KS60}
{\sc J.\ G.\ Kemeny \& J.\ L.\ Snell} (1960)
{\it Finite Markov chains.\/}
van Nostrand, Princeton~NJ.
%
\bibitem{KS61}
{\sc J.\ G.\ Kemeny \& J.\ L.\ Snell} (1961)
Finite continuous time Markov chains.
{\it  Teor.\ Verojatnost.\ i Primenen.\/} {\bf 6}, 101--105.
%
\bibitem{Khalil02}
{\sc H.\ K.\ Khalil} (2002)
{\it Nonlinear systems.\/}  3rd Edn, Prentice Hall, New Jersey.
}

\bibitem{Lind02}
{\sc T.\ Lindvall} (2002)
{\it Lectures on the coupling method\adbg{.}\/} \adbg{2nd Edn,}
Dover Publications,  Mineola~NY.
 
\ignore{
\bibitem{Pres83}
{\sc E.\ L.\ Presman} (1983)
On the approximation of binomial distributions by means of infinitely divisible ones.
{\it Teor.\ Verojatnost.\ i  Primenen.\/} {\bf 28}, 393--403. 
}

\bibitem{ReiRoe}
{\sc G.\ Reinert \& A.\ R\"ollin} (2009) 
Multivariate normal approximation with 
Stein's method of exchangeable pairs under a general linearity condition. 
{\it Ann.\ Probab.\/} {\bf 37}, 2150--2173. 

\bibitem{RinRot96}
{\sc Y.\ Rinott \& V.\ Rotar} (1996)
A multivariate CLT for local dependence with $n^{-1/2}\log n$ rate and
applications to multivariate graph related statistics.
{\it J.\ Multivariate Anal.\/} {\bf 56}, 333--350.

\ignore{
\bibitem{RR96}
{\sc G.\ O.\ Roberts \& J.\ S.\ Rosenthal} (1996)
Quantitative bounds for convergence rates of continuous time Markov processes.
{\it Elec.\ J.\ Probab.\/}~{\bf 1}, Paper no.~9, 21pp.
%
\bibitem{Roe05}
{\sc A.\ R\"ollin} (2005)
Approximation of sums of conditionally independent random 
variables by the translated Poisson distribution. 
{\it Bernoulli\/} {\bf 11},  1115--1128.
%
\bibitem{Roe07}
{\sc A.\ R\"ollin} (2007) 
Translated Poisson approximation using exchangeable pair couplings. 
{\it Ann.\ Appl.\ Probab.\/} {\bf 17},  1596--1614.
}

\bibitem{RR14}
{\sc A.\ R\"ollin \& N.\ Ross}  (2015)
Local limit theorems via Landau--Kolmogorov inequalities.
{\it Bernoulli\/} {\bf 21}, 851--880.

\ignore{
\bibitem{SB10} 
{\sc S. N. Socoll \& A. D. Barbour}  (2010) 
Translated Poisson approximation to equilibrium distributions of Markov population processes.
{\it Meth.\ Comp.\ Appl.\ Probab.\/} {\bf 12}, 567--586.
}

\bibitem{Stein86}
{\sc C.\ Stein} (1986)
{\it Approximate computation of expectations.\/}
IMS Lecture Notes {\bf 7}, Hayward, California.

\ignore{
\bibitem{Tropp}
{\sc J.\ A.\ Tropp} (2015)
Integer factorization of a positive-definite matrix. 
{\it SIAM J.\ Discrete Math.\/} {\bf 29}, 1783--1791. 
}

\end{thebibliography}
\end{document}